\documentclass[a4paper,12pt]{amsart}

\usepackage[hang,flushmargin]{footmisc}

\usepackage{graphicx,rotating}
\usepackage{floatrow}

\usepackage{blindtext}
\usepackage{siunitx}
\usepackage[font=footnotesize,labelfont=footnotesize]{caption}
\usepackage{subcaption}
\usepackage[utf8]{inputenc}
\usepackage[toc]{appendix}
\usepackage{twoopt}
\usepackage{float}
\usepackage{varwidth}
\newsavebox\tmpbox
\newfloatcommand{capbtabbox}{table}[][\FBwidth]

\usepackage{amssymb}
\usepackage{amsthm}
\usepackage{amsmath,a4wide}
\numberwithin{equation}{section}
\usepackage{hyperref}
\usepackage{mathtools}
\usepackage{amsfonts}
\usepackage{verbatim}
\usepackage{listings}
\usepackage{booktabs}
\usepackage{multirow}
\usepackage{marvosym}
\usepackage{verbatim}

\usepackage{enumerate}
\usepackage[shortlabels]{enumitem}
\usepackage{pdfsync}
\mathtoolsset{showonlyrefs}

\theoremstyle{plain}
\newtheorem{theorem}{Theorem}[section]
\newtheorem*{theorem*}{Theorem}
\theoremstyle{plain}

\theoremstyle{plain}
\newtheorem{lemma}[theorem]{Lemma}
\theoremstyle{plain}

\theoremstyle{definition}

\theoremstyle{remark}
\newtheorem*{remark}{Remark}
\theoremstyle{remark}
\newtheorem*{remarks}{Remarks}

\theoremstyle{definition}
\newtheorem{example}{Example}[section]
\theoremstyle{plain}

\theoremstyle{definition}

\newcommand{\R}{\mathbb{R}}

\newcommand{\Z}{\mathbb{Z}}
\newcommand{\N}{\mathbb{N}}
\newcommand{\Lt}[1][d]{L^2(\R)}

\newcommand{\indicator}{\raisebox{2pt}{$\chi$}}

\renewcommand{\d}{\delta}

\newcommand{\pw}{PW_b^2(g, \R)}
\newcommand{\pwi}{PW_b^2(g, I)}
\newcommand{\owb}{orthonormal Wilson basis}
\newcommand{\vbw}{variable bandwidth}
\newcommand{\tf}{time-frequency}

\newcommandtwoopt{\xarrow}[2][0.5cm][0]{\mathrel{\rotatebox[origin=c]{#2}{$\xrightarrow{\rule{#1}{0pt}}$}}}

\DeclareMathOperator{\supp}{supp}
\DeclareMathOperator{\spn}{span}

\DeclareMathOperator{\Rp}{Re}

\DeclarePairedDelimiter\abs{\lvert}{\rvert}
\DeclarePairedDelimiter\norm{\lVert}{\rVert}
\DeclarePairedDelimiter\ceil{\lceil}{\rceil}
\DeclarePairedDelimiter\floor{\lfloor}{\rfloor}

\usepackage{xcolor}

\title[Variable bandwidth via Wilson bases]{Variable bandwidth via Wilson bases}

\author[B.\ Andreolli]{Beatrice Andreolli}
\address{Faculty of Mathematics, University of Vienna \newline Oskar-Morgenstern-Platz 1, 1090 Vienna, Austria}
\email{beatrice.andreolli@univie.ac.at}
\author[K.\ Gröchenig]{Karlheinz Gröchenig}
\email{karlheinz.groechenig@univie.ac.at}

\thanks{Beatrice Andreolli was supported by the Austrian Science Fund (FWF) projects P31887-N32 and P33217. Karlheinz Gröchenig was supported by the Austrian Science Fund (FWF) project P31887-N32.}

\begin{document}
\maketitle
\begin{abstract}
    We introduce a new concept of variable bandwidth that is based on the frequency truncation of Wilson expansions. For this model we derive sampling theorems, a complete reconstruction of $f$ from its samples, and necessary density conditions for sampling. Numerical simulations support the interpretation of this model of variable bandwidth. In particular, chirps, as they arise in the description of gravitational waves, can be modeled in a space of variable bandwidth. 
\end{abstract}

\section{Introduction}
Variable bandwidth is a natural, but vexing concept in signal processing that has sparked the curiosity of engineers, physicists, and mathematicians. The difficulty of defining an adequate notion of \vbw\ lies in the nature of the Fourier transform and the uncertainty principle.  

A function $f \in L^2(\R)$ is said to have bandwidth $\Omega > 0$, if the number $\Omega$ is the maximal frequency contributing to $f$, in other words, its Fourier transform vanishes outside the interval $[-\Omega, \Omega]$. The associated function space is the Paley-Wiener space 
\begin{equation} \label{eq:pwc}
    PW_{\Omega} = \{f \in L^2(\R) : \supp(\hat{f})  \subseteq [-\Omega, \Omega]\}.
\end{equation}
By definition the bandwidth is a global quantity that requires the knowledge of $f$ on all of $\R $. In addition, functions in the Paley-Wiener space are entire functions of time, whereas real-world signals typically have a finite duration and also consist of bursts or pulses of varying duration, intensity, and frequency. Yet on an experiential level, \vbw\ obviously exists; for instance, in music we can perceive clearly the highest note, i.e., the local bandwidth, and its temporal changes, i.e., \vbw.

On a scientific level, there have been many attempts to define and work with \vbw, see our discussion below, but there is no consensus about the right definition.

In this work we propose a new approach to variable bandwidth  with
\tf\ methods. Time-frequency analysis offers  a natural toolbox for
variable bandwidth,  since its objective is roughly the analysis of
functions via joint \tf\ representations that combine the temporal
behavior $f$ and the frequency behavior $\hat{f}$ in one picture. Our
idea is to start with a \emph{discrete} \tf\ representation that
allows us to represent any $f$ as a series expansion of \tf\ atoms
with a clear localization both in time and in frequency.  
We may then prescribe a time-varying frequency truncation and in this way end up with a space of given \vbw. 

This idea goes back to Aceska and Feichtinger \cite{af11,af12} who used a time-varying frequency cut-off of the short-time Fourier transform. The resulting function spaces, however, coincide with the standard Sobolev spaces, and therefore lack the characteristic properties of Paley-Wiener space, and they admit neither a sampling theorem nor a critical density. 
Our  new  idea is to use a discrete version of the spaces of Aceska and Feichtinger and to replace the (continuous) short-time Fourier transform with a frequency truncation of a Wilson expansion. Our goal is then to show that the new spaces behave indeed like the classical Paley-Wiener space and prove sampling theorems and the existence of a critical sampling density.

The underlying \tf\ tool is an orthonormal basis of $L^2(\R )$ with a built-in \tf\ structure, a so-called Wilson basis. A Wilson basis is an orthonormal basis of $L^2(\R)$ of the form
\begin{align}\label{WB}
    &\psi_{n,l}(x)=
    \begin{cases}
        g(x-n), & n\in\Z, l = 0,  \\
        \frac{1}{\sqrt{2}}(e^{2 \pi i lx}+(-1)^{l+n}e^{-2 \pi i
          lx})g(x-n/2), & n,l\in \Z, l > 0. 
    \end{cases}
\end{align}
Daubechies, Jaffard and Journé \cite{djj91} showed that such
orthonormal bases exist and that the window function $g\in L^2(\R )$
may be chosen to be smooth and compactly supported. Then a basis
element $\psi_{n,l}$ is supported in a neighborhood of time $n$ and
possesses a frequency in a band near $\pm l$.  

\vspace{2mm}

\textbf{A new model for variable bandwidth.}
Now let $\{\psi_{n,l}\}_{n \in \Z, l \in \N}$ be an orthonormal Wilson basis as in \eqref{WB} and let $b : \Z \to \N$ be a bounded positive sequence representing the time-varying frequency truncation. 
Then we define a \vbw\ Paley-Wiener-type space as follows:
\begin{equation}\label{def:PW}
    PW_b^2(g,\R) = \Big \{f \in L^2(\R) : f = \sum_{n \in \Z}\sum_{l=0}^{b(n)}c_{n,l}\psi_{n,l}, \text{ } c \in \ell^2(\Z \times \mathbb{N}) \Big\}.
  \end{equation}
This is our main object of investigation. 
If $g$ has compact support, then the restriction $f|_{[n/2-1/2,n/2+1/2]}$ is approximately a trigonometric polynomial of degree $b(n)$, i.e., the maximal frequency on $[n/2-1/2,n/2+1/2]$ is expected to be $b(n)$.
 While this intuition is basically correct, we will need to refine it because the interval $[n/2-1/2,n/2+1/2]$ contains several translates $\psi _{n',l}$ for $n'\neq n$.

Wilson bases have proved to be very useful in time-frequency analysis since they provide a way to overcome the barrier posed by the Balian-Low theorem~\cite{bhw98}. 
Wilson bases are a popular tool in time-frequency analysis and signal processing to study function spaces time-frequency localization and non-linear approximation \cite{fgw92,gs00}. More importantly, Wilson bases have gained visibility and popularity for a general scientific audience because of their role in the detection of gravitational waves in 2015~\cite{aaa16, aa16, cjm16}.  
Gravitational waves are ripples in space-time caused by violent cosmic events like colliding black holes and supernovae and were predicted by Albert Einstein~\cite{Einst1,Einst2}. For the collision and merging of two neutron stars, the analytic form of the gravitational wave was calculated
in~\cite{bl14,bdiww95,bd99} to be of the form
\begin{equation} \label{eq:chi}
     s(t)=c|t-t_0|^{-\frac{1}{4}}\cos(\omega|t-t_0|^{\frac{5}{8}} +\varphi), \quad \text{ for } t<t_0,
\end{equation}
where $c$ is a constant, $\omega \gg 1$ and $t_0$ is the time of coalescence. Since $\omega|t-t_0|^{\frac{5}{8}} = \omega|t-t_0|^{-\frac{3}{8}} \, \abs{t-t_0}$, the quantity $\omega|t-t_0|^{-\frac{3}{8}}$ represents the (maximal)  frequency at time $t$. In the language of signal processing such a function is called  a \emph{chirp}, and is considered a prototype of a signal with \vbw. 

Essentially, the detection of gravitational waves amounts to the extraction of a chirp buried inside a noisy signal. 
In 2012, Necula, Klimenko, and Mitselmakher \cite{kmn12} proposed to use Wilson bases for gravitational wave detection and developed an algorithm that maps time series data to the time-frequency plane using a Wilson basis. This algorithm was successfully applied to gravitational data to detect the merger of two black holes and other astrophysical data.  

Let us emphasize that in this application the shape of the local bandwidth $ \omega|t-t_0|^{-\frac{3}{8}}$ is known in advance and can thus be incorporated into a suitable signal space. We argue and will demonstrate numerically  that a \vbw\ Paley-Wiener space $\pw$ is such a suitable space. 

\vspace{2mm}

\textbf{Results.} 
We will show that $\pw $ behaves in many ways like the classical Paley-Wiener space \eqref{eq:pwc}. In analogy to the Shannon-Whittaker-Kotelnikov sampling theorem and its variations, we will prove a sampling theorem for the variable bandwidth space $PW_b^2(g,\R)$. Secondly, we will determine a critical sampling density for $\pw $ that represents a necessary condition for sampling and can be seen as an information-theoretic quantity associated to
$\pw$. For the classical Paley-Wiener space, these questions have been studied for many decades~\cite{se04,un00} with the deepest results due to Beurling~\cite{be89} and Landau~\cite{la67}. These results have set the standard of how to approach sampling in a new function space.

In particular, sampling theorems are nowadays formulated as sampling inequalities. For this, we recall that a set $\Lambda \subset \R$ is called a set of (stable) sampling for $PW_b^2(g,\R)$, if there exist constants $A,B >0$ such that 
\begin{equation}\label{ineq:ss}
    A\norm{f}_2^2 \leq \sum_{\lambda \in \Lambda}|f(\lambda)|^2 \leq
    B\norm{f}_2^2 \qquad \text{ for all } f \in PW_b^2(g
    ,\R)\, .
\end{equation}
The sampling inequality \eqref{ineq:ss} indicates that the entire information carried by the function is captured by the evaluation of the function at the samples. Furthermore, a sampling entails automatically several reconstruction methods of a function from its samples. 

We start with  a necessary density condition for sampling in $\pw$ that is valid for a very general class of windows.   
\begin{theorem}[Necessary density condition for sampling] \label{tint1}
    Assume that  $g\in \mathcal{C}(\R)$ is a  real-valued, even function such that $|g(x)| \leq C(1+|x|)^{-1-\epsilon}$ for some $C >0$ and $\epsilon > 0$, and that $g$ generates an orthonormal Wilson basis.
    Let $b=(b(n))_{n\in \Z} $ be a bounded sequence such that
    $b(n)\geq 1$ for every $n \in \Z$ and let $\pw$ be the  associated variable bandwidth space.
    If  $\Lambda \subseteq \R$ is a set of sampling for $PW_b^2(g,\R)$, then
    \begin{equation}
        D^-(\Lambda) \coloneqq \liminf_{r \to \infty}\inf_{x \in \R} \frac{\#(\Lambda \cap [x-r,x+r])}{2r} \geq 1 + \overline{b}
    \end{equation}
    where $\overline{b} =\liminf_{r \to \infty} \inf_{x \in \R}\frac{1}{2r} \sum_{\frac{n}{2} \in [x-r,x+r]} b(n)$.
\end{theorem}
In this result, $\overline{b}$ represents a sort of average bandwidth. The necessary condition says that on average are needed at least $1+\bar{b}$ samples to recover $f\in \pw $.

On the constructive side, our main theorem provides a sufficient condition for sampling for $PW_b^2(g,\R)$. In the absence of more structure, such as analyticity, we use a local maximum gap between consecutive samples to formulate the sufficient condition. 
\begin{theorem}[Sufficient condition for sampling] \label{tint2}
    Let $g \in \mathcal{C}^1(\R)$ be  a real-valued, even function with $\supp(g) \subseteq [-m, m]$ that generates an orthonormal Wilson basis. 
    Let $PW_b^2(g,\R)$ be the space of variable bandwidth defined in \eqref{def:PW}. Label $\Lambda  \subseteq \R$ as
    \begin{equation}
        \Lambda \coloneqq \Big \{x_{\frac{k}{2},j} = k/2 +\eta_{\frac{k}{2},j} : k \in \Z, \text{ } \eta_{\frac{k}{2},j}\in [0,1/2), \text{ } j=1, \dots , j_{\max}(\tfrac{k}{2})  \Big\}.
    \end{equation}
    Let  $\delta_\frac{k}{2} = \max_{j=1,\dots, j_{\max}(\frac{k}{2})}(x_{\frac{k}{2},j+1}-x_{\frac{k}{2},j})$ be the maximum gap on $[\tfrac{k}{2}, \tfrac{k+1}{2}]$ and assume that $x_{\frac{k}{2},1}-\tfrac{k}{2} \leq \frac{1}{2}\delta_{\frac{k}{2}}$ and $\frac{k+1}{2}-x_{\frac{k}{2},j_{\max(\frac{k}{2})}} \leq \frac{1}{2}\delta_{\frac{k}{2}}$.
    Define 
    \begin{equation}
        D \coloneqq (4m)\cdot\max\{2\pi\norm{g}_\infty,\norm{g'}_\infty\}
    \end{equation}
    and 
    \begin{equation}
        \mu_{\frac{k}{2}} \coloneqq \max_{n =(k-2m,k+2m+1)\cap \Z} \Big(b(n)+1\Big).
    \end{equation}
    If for every $ k \in \Z$
    \begin{equation}
        \d_{\frac{k}{2}} < \frac{\pi}{\mu_{\frac{k}{2}} D},
    \end{equation}
    then $f \in PW_b^2(g,\R)$ can be reconstructed completely from  the samples in $\Lambda$.

    If in addition  $\Lambda$  is separated, then $\Lambda$ is a set of sampling for $\pw $ and a sampling inequality \eqref{ineq:ss} holds.
\end{theorem} 
The formulation confirms the initial intuition that the parameter $b(k)$ measures the local bandwidth of $f\in \pw$ near $k/2$. However, we learned that also the influence of the neighboring intervals must be included and that the parameter $\mu _{k/2}$ is a more accurate measure for the local bandwidth. For an illustrative example we refer to the discussion in Section \ref{Sec4}.

The second important aspect is the local nature of $\pw $. Since $\supp (g) $ is compact, the restriction $\pw \big| _I$ to an interval $I$ is in fact finite-dimensional. This is quite in contrast to the classical Paley-Wiener space of entire functions and greatly facilitates the numerical treatment of sampling in $\pw $.

For the proof we will resort to the adaptive weights method described in \cite{fg94,gr92}.  

\vspace{2mm}

\textbf{Other concepts of \vbw\ in the literature.} In the literature one finds several possible approaches to variable bandwidth. While they all have a common goal, the mathematical methods are entirely disjoint from our approach. 

One of the first attempts to capture  variable bandwidth in signal processing involves warping functions \cite{bpp98, cpl85, ho68, sg12, wo07}. 
Given a homeomorphism $\gamma: \R \to \R$,  a so-called warping function, a function $f$ possesses variable bandwidth with respect to $\gamma$ if  $f = g \circ \gamma$ for a bandlimited function $g \in L^2(\R)$ with $\supp(\hat{g}) \subseteq [-\Omega, \Omega]$. The derivative $1/\gamma'(\gamma^{-1}(x))$ of the warping function is interpreted as the local bandwidth of $f$ at $x$.

A recent definition of variable bandwidth is based on the spectral theory of a Sturm-Liouville operator $A_pf = - \frac{d}{dx}(p(x)\frac{d}{dx})f$ on $L^2(\R)$, where $p > 0$ is a strictly positive function \cite{cgk23, gk17, gk21}. Denote by $c_{\Lambda}(A_p)$ the spectral projection of $A_p$ corresponding to $A_p$ in $\Lambda \subset \R^+$, then the range of this projection-valued measure $c_\Lambda(A_p)$ is the space of functions of variable bandwidth with spectral set in $\Lambda$. For this space of variable bandwidth were proved a (nonuniform) sampling theorem and necessary density conditions.  
The main insight is that, for a spectrum $\Lambda = [0, \Omega]$, a function of variable bandwidth behaves like a bandlimited function with local bandwidth $(\Omega/p(x))^{1/2}$ in a neighborhood of $x \in \R$. 

Kempf and his collaborators used a procedural idea of variable bandwidth~\cite{hk10, ke00, km08}. A formal definition  concept was then  presented in~\cite{mk18}. Essentially, they generated a symmetric operator and a related reproducing kernel Hilbert space based on a sequence of sampling points. This space was designated as a variable bandwidth space. By parameterizing the self-adjoint extensions of the operator, they obtained several sampling theorems at the critical density.

\vspace{2mm}

\textbf{Numerical simulations.} Our theoretical analysis of sampling in $\pw $ is complemented by detailed numerical simulations. We will consider several numerical scenarios and classes of sampling sets. We study (i) the reconstruction of a compactly supported function in $\pw $, and (ii) the approximation of the restriction to an interval of an arbitrary function. In this case boundary effects occur and have to be dealt with. As a highlight, we reconstruct a chirp similar to \eqref{eq:chi} from its samples.

\vspace{2mm}

\textbf{Outlook.} Wilson bases are by no means the only discrete \tf\ representations that are suitable for our approach. Alternative definitions of \vbw\ can be defined by using a variable frequency truncation of  a Gabor frame expansion or by replacing a Wilson basis with local Fourier bases. The achievable results are similar in spirit, but these alternative set-ups also pose different challenges (for instance, well localized Gabor frames are not linearly independent). See~\cite{an24} for more information. At this time, it is not clear which version of \vbw\ is the most convenient one.

Finally, from the point of view of signal processing and data
analysis, it is essential to determine the local bandwidth parameters
$b(n)$ based on given samples. This problem is one of  parameter
estimation and  is usually treated separately and needs to be dealt with even for the classical Paley-Wiener space. 

The paper is organized as follows: Section \ref{Sec2} introduces the main tools required, namely the  characterization of a Wilson basis, a version of the necessary density condition theorem for reproducing kernel Hilbert spaces \cite{fghkr17}, the proof that $PW_b^2(g,\R)$ is a reproducing kernel Hilbert space, and several standard inequalities.
Section \ref{Sec3} contains the proof of the necessary density condition (Theorem~\ref{tint1}) for sampling for the space of variable bandwidth $PW_b^2(g, \R)$. In Section \ref{Sec4} we provide the proof of the sufficient condition for sampling in $PW_b^2(g, \R)$ (Theorem~\ref{tint2}). In Section \ref{Sec5} we discuss the numerical implementation of the \vbw\ model and present the numerical reconstruction results. 

\section{Preliminaries}\label{Sec2}
In this section we recall the basic properties of a Wilson basis and the tools we need to investigate sampling.

By defining the usual translation and modulation operator as $T_nf(t) = f(t - n)$ and\\ 
$M_l f(t) = f(t)e^{2 \pi i l t}$, we can rewrite the entire collection of functions that forms a Wilson basis \eqref{WB} as
\begin{equation}\label{WB:2ndef}
    \tilde{\psi}_{n,l} = d_l (M_l + (-1)^{l+n}M_{-l})T_{\frac{n}{2}}g, \quad (n,l) \in \Z^2, l\geq0,
\end{equation}
where $d_0= \frac{1}{2}$ and $d_l = \frac{1}{\sqrt{2}}$, $l\geq 1$. In this form $\tilde{\psi}_{n,l} = \psi _{n,l}$ for $l\geq 1$ and $n\in \Z $, and $\tilde{\psi}_{2n,0} = \psi _{n,0} = T_ng$, and $\tilde{\psi}_{2n+1,0}= 0$. 

The following result from \cite[Corollary 8.5.4]{gr01} gives a characterization of Wilson bases. 
\begin{theorem}\label{WB_char}
    Assume $g \in L^2(\R)$ to be even and real-valued. Then the following are equivalent:
    \begin{enumerate}[label=(\roman*)]
        \item The Wilson system $\{\psi_{n,l}\}_{n \in \Z, l \in \mathbb{N}}$ as defined by \eqref{WB} is an orthonormal basis for $L^2(\R)$.
        \item $\norm{g}_2 = 1$ and 
        the Gabor system $\{T_{n/2}M_lg : n,l \in \Z \}$
        is a tight frame for $L^2(\R)$, i.e., there exists a constant
        $A>0$ such that for all $f \in L^2(\R)$ 
        \begin{equation}\label{eq:TightFrame}
        \sum_{n \in \Z}\sum_{l \in \Z}|\langle f, T_{n/2}M_lg\rangle|^2 = A\norm{f}_2^2.
\end{equation}
        \item $\sum_{n \in \Z}g(x-k-\frac{n}{2})\overline{g}(x-\frac{n}{2}) = 2 \delta_{k0}$ a.e.
    \end{enumerate}
\end{theorem}

Orthonormal Wilson bases  exist in abundance and the window function $g$ may be chosen to be smooth with compact support or may be endowed with other desirable properties~\cite{auscher94,djj91}.

To study necessary conditions for sampling in $\pw $, we will make use of the general theory of sampling  in reproducing kernel Hilbert spaces and the necessary density conditions derived in~\cite{fghkr17}.  
We recall that a closed subspace $\mathcal{K} \subseteq L^2(\R)$ is a reproducing kernel Hilbert space if the point evaluation $f\mapsto f(x)$ is a bounded linear functional on $\mathcal{K}$ for all $x\in \R$. In this case there exists a family $k_x\in \mathcal{K}, x\in \R,$ such that $f(x) = \langle f, k_x \rangle$.  The function $k(x,y) = \overline{k_x(y)}$ is called the reproducing kernel of $\mathcal{K}$.
We formulate a simplified version of \cite[Thm.~1.1]{fghkr17} that is adapted to $\R $ with Lebesgue measure. 

\begin{theorem}[\cite{fghkr17}]\label{NDC:Thm}
    Let $\mathcal{K} \subseteq L^2(\R)$ be a reproducing kernel Hilbert space with a reproducing kernel $k(x,y)$ satisfying
    \begin{enumerate}[label=(\roman*)]
        \item $\inf_{x \in \R}k(x,x)>0$ and
        \item off-diagonal decay of the form $|k(x,y)| \leq
          N(1+|x-y|)^{-1-\epsilon}$ for all $x,y \in \R$ and some $\epsilon > 0$.
    \end{enumerate} 
    If for $\Lambda \subset \R$ there exist $A,B >0$ such that for every $f \in \mathcal{K}$
    \begin{equation}\label{samp_ineq}
        A\norm{f}^2 \leq \sum_{\lambda \in \Lambda} |f(\lambda)|^2 \leq B\norm{f}^2,
    \end{equation}
    then
    \begin{equation*}
        D^-(\Lambda) \coloneqq \liminf_{r \to \infty} \inf_{x \in \R} \frac{\#(\Lambda \cap [x-r,x+r])}{2r} \geq \liminf_{r \to \infty} \inf_{x \in \R} \frac{1}{2r}\int_{x-r}^{x+r}k(y,y)dy.
    \end{equation*}
\end{theorem}
We recall that the sampling inequality \eqref{samp_ineq} is equivalent of $\{k_\lambda\}_{\lambda \in \Lambda}$ being a frame for $\mathcal{K}$. The number $D^-(\Lambda) \coloneqq \liminf_{r \to \infty} \inf_{x \in \R} \frac{\#(\Lambda \cap [x-r,x+r])}{2r}$ is the Beurling density of $\Lambda$.

To apply Theorem~\ref{NDC:Thm}, we first show that the variable bandwidth space $\pw $  is a reproducing kernel Hilbert
space. 
\begin{lemma} \label{PW_RKHS}
    The space $PW_b^2(g, \R)$ of variable bandwidth defined in \eqref{def:PW} is a reproducing kernel Hilbert space with reproducing kernel
    \begin{equation}\label{eq:RepKer}
        k(x,y) = \sum_{n \in \Z}\sum_{l = 0}^{b(n)}\overline{\psi_{n,l}(y)}\psi_{n,l}(x).
    \end{equation}
\end{lemma}

\begin{proof}
    Recall from the definition \eqref{def:PW} of $PW_b^2(g,\R)$ that $b(n) \leq B < \infty$ for every $n \in \Z$. 
    Since $\{\psi_{n,l}\}_{n \in \Z, l \in \mathbb{N}}$ is an orthonormal basis for $L^2(\R)$, then $PW_b^2(g,\R) = \spn\{
    \psi_{n,l}: n \in \Z, \text{ } l = 0,\dots,b(n) \}$ is a closed subspace of $L^2(\R)$.
    
    We  show that at each point in $\R$ the linear functional $f\mapsto f(x)$  is bounded i.e. there exists $M > 0$ such that $|f(x)| \leq M \norm{f}_2$ for every $f \in PW_b^2(g, \R)$ and $\forall x \in \R$.
    Let $x \in \R$ and $f = \sum_{n \in \Z}\sum_{l = 0}^{b(n)} \langle f, \psi_{n,l} \rangle \psi_{n,l} \in PW_b^2(g, \R)$, then applying the Cauchy-Schwarz inequality, we obtain
    \begin{align}\label{ineq:bound_f}
        |f(x)| 
        & = \Big|\sum_{n \in \Z}\sum_{l = 0}^{b(n)}\langle f,\psi_{n,l} \rangle \, \psi_{n,l}(x)\Big| \\
        & \leq \Big(\sum_{n \in \Z}\sum_{l = 0}^{b(n)} \Big|\langle  f , \psi_{n,l} \rangle\Big|^2\Big)^{1/2} \Big(\sum_{n \in \Z}\sum_{l = 0}^{b(n)} \Big|\psi_{n,l}(x)\Big|^2\Big)^{1/2}.
    \end{align}
    Since $\{\psi_{n,l}\}_{n \in \Z, l \leq b(n)}$ is an orthonormal basis for $PW_b^2(g,\R)$, we have that
    \begin{equation*}
        \Big(\sum_{n \in \Z}\sum_{l = 0}^{b(n)} \Big|\langle  f , \psi_{n,l} \rangle\Big|^2\Big)^{1/2}
        = \norm{f}_2.
    \end{equation*}
    To bound the second term of \eqref{ineq:bound_f}, we apply the definition of a Wilson basis  and we obtain
    \begin{equation*}
        \sum_{n \in \Z}\sum_{l = 0}^{b(n)} \Big|\psi_{n,l}(x)\Big|^2 
        = \sum_{n \in \Z} \Big|\psi_{n,0}(x)\Big|^2 + \sum_{n \in \Z}\sum_{l = 1}^{b(n)} \Big|\psi_{n,l}(x)\Big|^2
        = \sum_{n \in \Z} |g(x-n)|^2 + \sum_{n \in \Z}\sum_{l = 1}^{b(n)} \Big|\psi_{n,l}(x)\Big|^2.
    \end{equation*}
    Recall from the characterization of Wilson basis the equivalence (iii) of Theorem \ref{WB_char}. Then the first summand can be easily bounded as follows:
    \begin{equation}\label{IneqBoundIntTranslates}
        \sum_{n \in \Z} |g(x-n)|^2 
        \leq \sum_{n \in \Z} \Big|g\Big(x-\frac{n}{2}\Big)\Big|^2 = 2.
    \end{equation}
    The second summand can be bounded using the definition of the Wilson basis, the fact that $b(n) \leq B$ for every $n \in \Z$ and again equivalence (iii) of Theorem \ref{WB_char}. Therefore,
    \begin{equation}\label{IneqBoundHalfTranslates}
        \sum_{n \in \Z}\sum_{l = 1}^{b(n)} \Big|\psi_{n,l}(x)\Big|^2 
        \leq 2  \sum_{n \in \Z}\sum_{l = 1}^{b(n)} \Big|g\Big(x-\frac{n}{2}\Big)\Big|^2
        \leq 4B.
    \end{equation}
    Then $|f(x)| \leq \sqrt{2(1+2B)}\norm{f}_2$ for every $x \in \R$ and for every $f \in PW_b^2(g, \R)$. Hence,  $PW_b^2(g, \R)$ is a reproducing kernel Hilbert space.
    
    Moreover, since $\{\psi_{n,l}\}_{n \in \Z, l \leq b(n)}$ is an orthonormal basis for $PW_b^2(g, \R)$, the reproducing kernel is given by
    \begin{equation}
        k(x,y) = \sum_{n \in \Z}\sum_{l = 0}^{b(n)}\overline{\psi_{n,l}(y)}\psi_{n,l}(x).
    \end{equation}
\end{proof}

The following versions of Wirtinger's inequality and Bernstein's inequality will be useful tools to prove the sufficient condition for sampling.

\begin{lemma}[Wirtinger's inequality]\label{Wirt_ineq}
    If $f,f'\in L^2(a,b)$, $a<c<b$, and $f(c)=0$, then
    \begin{equation}\label{WirtIneq}
        \int_a^b|f(x)|^2dx \leq \frac{4}{\pi^2} \max\{(b-c)^2,(c-a)^2\} \int_a^b|f'(x)|^2dx.
    \end{equation}
\end{lemma}
Lemma \ref{Wirt_ineq} follows from \cite[p.184]{hlp52},  by applying a change of variables.
Wirtinger's inequality and its variations are often used in the proofs of sampling theorems, see \cite{asr17,gr92,wa87}.

Bernstein's inequality provides a bound for the derivative of a  trigonometric polynomial.
\begin{lemma}[Bernstein's inequality]\label{Bern_ineq}
    Let $P$ be a trigonometric polynomial of degree $n$ i.e. $P(x) = \sum_{|k| \leq n} c_ke^{2 \pi i k x}$.  Then
    \begin{equation}\label{eq:BernIneq}
        \norm{P'}_2 \leq 2 \pi n\norm{P}_2.
    \end{equation}
\end{lemma}

\section{Necessary density condition}\label{Sec3}
In this section, we will prove necessary density conditions for the spaces of variable bandwidth generated by a Wilson basis. Since $\pw$ is modeled after the Paley-Wiener space of bandlimited functions, we expect that there exists a critical density for $\pw $ as well. A critical density is a fundamental limitation for the reconstruction  of a function from its samples and quantifies how many samples are required. Furthermore, the  critical density is a benchmark against which one can measure constructive approaches to a reconstruction. 

We start with a simple result for Wilson bases with a compactly supported window. 

\begin{theorem}\label{NCW:compact}
    Let $g \in \mathcal{C}(\R)$ be a  real-valued, even function  with
    $\supp(g) 
    \subseteq [-m,m]$ that generates an \owb . Let $\Lambda \subseteq \R$ be a set of sampling
    for $PW_b^2(g, \R)$ defined in \eqref{def:PW}. Then  $\Lambda \cap (\alpha,\beta)$ contains at least
    $\ceil{\beta-\alpha-2m}+\sum_{n/2 \in [\alpha+m,\beta-m]}b(n)$ points.
\end{theorem}
To show the result we use an easy argument that involves counting dimensions as in \cite{ag00}.

\begin{proof}
    Let $I = (\alpha, \beta)$ such that $|I| \geq 2m$ and define the
    local variable bandwidth space $ PW_b^2(g,I)$  by 
    \begin{align}\label{def:PW_I}
      PW_b^2(g,I) &= \{\psi_{n,l} : \supp (\psi _{n,l})\subseteq I\} \\
      & =\spn\{\psi_{n,0} : n \in [\alpha+m,\beta-m]\}\\
        &\qquad \cup \spn \{\psi_{n,l} : n/2 \in [\alpha+m,\beta-m], \text{ } l = 1,\dots,b(n)\}.
    \end{align}
    We make the following three observations:
    \begin{enumerate}[(i)]
        \item $PW_b^2(g,I)$ is a subspace of $PW_b^2(g,\R)$ and thus the sampling inequalities hold for $PW_b^2(g,I)$.
        \item Since the Wilson basis functions are compactly supported,  every $f \in PW_b^2(g,I)$ has support in $(\alpha, \beta)$.
        \item $PW_b^2(g,I)$ is finite-dimensional,  and 
            \begin{equation}\label{eq:dim_pwi}
                \dim(PW_b^2(g,I)) = \#([\alpha+m,\beta-m] \cap \Z) + \sum_{n/2 \in [\alpha+m,\beta-m]} b(n).
            \end{equation}
    \end{enumerate}
    Since for $f \in PW_b^2(g,I)$, we have
    \begin{equation*}
        \sum_{\lambda \in \Lambda \cap (\alpha, \beta)} \abs{f(\lambda)}^2 
        = \sum_{\lambda \in \Lambda} \abs{f(\lambda)}^2 \geq A \norm{f}_2^2,
    \end{equation*}
    the map $f \in PW_b^2(g,I) \to \{f(\lambda)\}_{\lambda \in \Lambda \cap (\alpha, \beta)}$ is one-to-one. It follows that
    \begin{equation}
        \dim (PW_b^2(g,I)) = \ceil{\beta-\alpha -2m} + \sum_{n/2 \in [\alpha+m,\beta-m]} b(n) \leq \#(\Lambda \cap (\alpha, \beta)).
    \end{equation}
\end{proof}

Next we  state a more general theorem for a more general class of windows satisfying a mild decay condition.  

\begin{theorem}\label{thm:NDCgeneral}
    Let $g \in \mathcal{C}(\R)$ be a  real-valued, even function, such
    that $|g(x)| \leq C(1+|x|)^{-1-\epsilon}$ for some $C >0$ and
    $\epsilon > 0$ and assume that $g$  generates an \owb. 
    Let $\Lambda \subseteq \R$ be a set of sampling for $PW_b^2(g,\R)$ with $b(n) \geq 1$ for every $n \in \Z$.  Then
    \begin{equation*}
        \liminf_{r \to \infty}\inf_{x \in \R} \frac{\#(\Lambda \cap [x-r,x+r])}{2r} \geq 1 + \overline{b}
    \end{equation*}
    where $\overline{b} =\liminf_{r \to \infty} \inf_{x \in \R}\frac{1}{2r} \sum_{\frac{n}{2} \in [x-r,x+r]} b(n)$.
\end{theorem}

\begin{proof}
    By Lemma \ref{PW_RKHS}, $PW_b^2(g,\R)$ is a reproducing kernel Hilbert space. In order to  apply the results on sampling in general reproducing kernel Hilbert spaces from \cite{fghkr17} in the form of Theorem \ref{NDC:Thm}, we only need to verify the conditions on the reproducing kernel. 

    \vspace{2mm}
    
    \textbf{Step 1.}  We show that the diagonal $k(x,x)$ of the  reproducing kernel is bounded below (condition (i) of  Theorem \ref{NDC:Thm}).
    After substituting  the definition of a Wilson basis in \eqref{eq:RepKer}, we write
    \begin{equation} \label{eq:c1b}
        k(x,x) = \sum_{n \in \Z} |g(x-n)|^2 + \frac{1}{2}\sum_{n \in \Z}\sum_{l = 1}^{b(n)} |e^{2 \pi i l x} +(-1)^{l+n}e^{-2 \pi i l x}|^2 \Big|g\Big(x-\frac{n}{2}\Big)\Big|^2.
    \end{equation}
    Since $g$ has compact support, the series are locally finite. The continuity of $g$ implies that $k(x,x)$ is continuous, and \eqref{eq:c1b} shows that $k(x,x)$ is periodic with period $1$. Consequently it suffices to show that  $k(x,x)>0$ on $[0,1)$.

    We use $b(n) \geq 1$ for every $n \in \Z$ and obtain
    \begin{align}
       k(x,x)  \geq \sum_{n \in \Z} |g(x-n)|^2 + \frac{1}{2}\sum_{n \in \Z} |e^{2 \pi i x} +(-1)^{1+n}e^{-2 \pi i x}|^2 \Big|g\Big(x-\frac{n}{2}\Big)\Big|^2.
    \end{align}
    By dividing the second sum into even and odd  parts and applying (iii) of Theorem \ref{WB_char}, we get
    \begin{align}
       k(x,x) & \geq \sum_{n \in \Z} |g(x-n)|^2 + 2\sum_{n \in \Z}
                \abs{\sin{2 \pi x}}^2 |g(x-n)|^2 + 2\sum_{n \in
                \Z}\abs{\cos{2 \pi
                x}}^2\Bigl|g\Bigl(x-n-\frac{1}{2}\Bigr)\Bigr|^2 \notag
      \\
       & = \sum_{n \in \Z} |g(x-n)|^2 (1+2\abs{\sin{2 \pi x}}^2) +  2\sum_{n \in \Z}\abs{\cos{2 \pi x}}^2\Bigl|g\Bigl(x-n-\frac{1}{2}\Bigr)\Bigr|^2 \label{ineq:traceRK}\\
       & \geq \min\{1+2\abs{\sin{2 \pi x}}^2,2\abs{\cos{2 \pi
         x}}^2\}\sum_{n \in
         \Z}\Bigl|g\Bigl(x-\frac{n}{2}\Bigr)\Bigr|^2 \notag \\
       & = 2\min\{1+2\abs{\sin{2 \pi x}}^2,2\abs{\cos{2 \pi x}}^2\}
         . \notag 
    \end{align}
    
    The only possible zeros of $k(x,x)$ for $x \in [0,1)$ are for $\cos{2 \pi x} = 0$, i.e, at $x=\frac{1}{4}$ or $x = \frac{3}{4}$.
    Since $g$ is even, $k(x,x)$ is even, and thus periodicity yields  $k(\frac{1}{4},\frac{1}{4}) = k(-\frac{1}{4},-\frac{1}{4}) = k(\frac{3}{4},\frac{3}{4})$  and $k(\frac{1}{4},\frac{1}{4}) + k(\frac{3}{4},\frac{3}{4}) = k(\frac{1}{4},\frac{1}{4}) + k(-\frac{1}{4},-\frac{1}{4})$.
    Applying (iii) of Theorem \ref{WB_char} and from the inequality \eqref{ineq:traceRK}, we obtain
    \begin{align}
        k\Bigl(\frac{1}{4},\frac{1}{4}\Bigr) + k\Bigl(-\frac{1}{4},-\frac{1}{4}\Bigr) 
        & \geq 3 \Bigl (\sum_{n \in \Z} \Bigl |g \Bigl(\frac{1}{4}-n\Bigr)\Bigr|^2 + \sum_{n \in \Z} \Bigl |g\Bigl(-\frac{1}{4}-n\Bigr)\Bigr|^2 \Bigr)\\
        & = 3 \Bigl (\sum_{n \in \Z} \Bigl |g \Bigl(\frac{1}{4}-n\Bigr)\Bigr|^2 + \sum_{n \in \Z} \Bigl |g\Bigl(\frac{1}{4}-n-\frac{1}{2}\Bigr)\Bigr|^2 \Bigr) = 6.
    \end{align}
    
    By symmetry $k(\frac{1}{4},\frac{1}{4}) = k(\frac{3}{4},\frac{3}{4}) = 3 \neq 0$.\\
    It follows that $k(x,x) > 0$ for all $x \in [0,1)$ and thus $\inf_{x \in \R}k(x,x) = \min_{x \in \R} k(x,x) > 0$.

    \vspace{2mm}
    
    \textbf{Step 2.} We verify the required decay condition of the kernel (condition (ii) of Theorem \ref{NDC:Thm}). 
    Let $x,y \in \R$. Recall that $b(n) \leq B$ for every $n \in \Z$ and $|\psi_{n,l}(x)| \leq \sqrt{2} g(x-n/2)$ for $l \neq 0$, then we have
    \begin{align}\label{inq:kern_decay}
        |k(x,y)| 
        = & \abs[\Big]{\sum_{n \in \Z}\sum_{l = 0}^{b(n)}\overline{\psi_{n,l}(y)}\psi_{n,l}(x)}\nonumber\\
        = & \abs[\Big]{\sum_{n \in \Z} \overline{\psi_{n,0}(y)}\psi_{n,0}(x) + \sum_{n \in \Z} \sum_{l = 1}^{b(n)}\overline{\psi_{n,l}(y)}\psi_{n,l}(x)}\nonumber\\
        \leq & \sum_{n \in \Z} \{ \abs{g(y-n)g(x-n)} + 2b(n) \abs{g(y-n/2)g(x-n/2)}\}\nonumber\\
        \leq & \sum_{n \in \Z} \{ \abs{g(y-n)g(x-n)} + 2B \abs{g(y-n/2)g(x-n/2)}\}.
    \end{align}
    To deal with these sums, we use a convolution estimate  of the weight function $(1 + |x|)^{-1-\epsilon}$ from ~\cite{gr04} and obtain
    \begin{align}
            \sum_{n \in \Z} |g(y-n)g(x-n)| 
            &\leq C^2 \sum_{n \in \Z} (1 + |y-n|)^{-1-\epsilon} (1 + |x-n|)^{-1-\epsilon}\\
            &\leq C^2 C_0  (1 + |x-y|)^{-1-\epsilon} \, ,
     \end{align}
    where the last inequality is precisely the statement of \cite[Lemma~2]{gr04}. 

    Similarly, the second term in \eqref{inq:kern_decay} is bounded by
    \begin{equation}
        \sum_{n \in \Z} \abs{g(y-n/2)g(x-n/2)} \leq \widetilde{C} (1+|x-y|)^{-1-\epsilon},
    \end{equation}
    so that we obtain the desired off-diagonal decay of the kernel.
    Conditions (i) and (ii) of Theorem \ref{NDC:Thm} are satisfied, and Theorem \ref{NDC:Thm} is applicable.

    \vspace{2mm}
    
    \textbf{Step 3.} It remains to compute the averaged trace $\frac{1}{2r} \int_{x-r}^{x+r}k(y,y)dy$ and let $r $ tend to $\infty$. We write $B_r(x) = [x-r,x+r]$ in the following. 

    Fix $\varepsilon > 0$ and choose $\rho > 0$, such that $\int_{B_\rho(0)}|\psi_{0,l}(x)|^2dx \geq 1 - \varepsilon$ for $ l = 0,\dots,B$. Then
    \begin{align}
        & \int_{B_\rho(n/2)}|\psi_{n,l}(x)|^2dx \geq 1 - \varepsilon \quad \text{for }l = 1,\dots,B 
        \quad \text{ and}\\
        & \int_{B_\rho(n)}|\psi_{n,0}(x)|^2dx \geq 1 - \varepsilon.
    \end{align}
    Using \eqref{IneqBoundHalfTranslates}, we can apply dominated convergence and interchange summation and integration and obtain
    \begin{equation*}
         \frac{1}{2r} \int_{B_r(x)}k(y,y)dy
         = \frac{1}{2r} \int_{B_r(x)}\sum_{n \in \Z}\sum_{l = 0}^{b(n)} |\psi_{n,l}(y)|^2dy
         = \frac{1}{2r}\sum_{n \in \Z}\sum_{l = 0}^{b(n)} \int_{B_r(x)} |\psi_{n,l}(y)|^2dy.
     \end{equation*}
    
    If $\frac{n}{2} \in B_{r-\rho}(x)$, then $B_\rho(\frac{n}{2}) \subseteq B_r(x)$. We estimate the average trace as follows:
    \begin{align*}
       \frac{1}{2r} \int_{B_r(x)}k(y,y)dy
       & \geq \frac{1}{2r} \sum_{n \in B_{r-\rho}(x)}\int_{B_r(x)} |\psi_{n,0}(y)|^2dy + \frac{1}{2r} \sum_{\frac{n}{2} \in B_{r-\rho}(x)}\sum_{l = 1}^{b(n)}\int_{B_r(x)} |\psi_{n,l}(y)|^2dy\\
       & \geq \frac{1}{2r} (1 - \varepsilon) \Bigl [\#(B_{r-\rho}(x) \cap \Z) + \sum_{\frac{n}{2} \in B_{r-\rho}(x)} b(n)\Bigr]\\
       & \geq \frac{1}{2r} (1 - \varepsilon) \Bigl [2(r-\rho) + \sum_{\frac{n}{2} \in B_{r-\rho}(x)} b(n)\Bigr]\\
       & = (1 - \varepsilon) \frac{r-\rho}{r} + \frac{1}{2r}(1 - \varepsilon) \sum_{\frac{n}{2} \in B_{r-\rho}(x)} b(n).\\
    \end{align*}
    Letting $r \to \infty$, we obtain
    \begin{equation}
        \liminf_{r \to \infty} \inf_{x \in \R}  \frac{1}{2r} \int_{B_r(x)}k(y,y)dy 
        \geq 1- \varepsilon + (1-\varepsilon) \liminf_{r \to \infty} \inf_{x \in \R} \frac{1}{2r} \sum_{\frac{n}{2} \in B_{r}(x)} b(n).
    \end{equation}
    Define $\overline{b} \coloneqq \liminf_{r \to \infty}\inf_{x \in
      \R} \frac{1}{2r}\sum_{\frac{n}{2} \in B_{r}(x)} b(n)$ as  the average bandwidth in $B_r(x)$. 
    Since $\varepsilon >0$ is arbitrary, Theorem \ref{NDC:Thm} yields 
    \begin{equation}
        D^-(\Lambda) \geq 1 + \overline{b} \, ,
    \end{equation}
    which finishes the proof.
\end{proof}

\begin{remarks}
    (i) Note that in the above proof the properties of Wilson bases play a crucial role. In particular, we have used the characterization  (iii) of Theorem \ref{WB_char} and the symmetry of the window function $g$. 
    
    (ii) It is easy to see  that for a compactly supported window $g$ Theorem \ref{NCW:compact} implies directly Theorem \ref{thm:NDCgeneral}.
\end{remarks}

\section{Sufficient condition for sampling}\label{Sec4}
With some necessary conditions for sampling in place, we now investigate the question under which conditions on a set a function in $PW_b^2(g,\R)$ can be reconstructed completely. In this context, we use a Wilson basis with a continuous, compactly supported window $g$. This condition guarantees the locality of the reconstruction. 

We will index the sampling set $\Lambda$ as follows
\begin{equation}\label{Lambda_ss}
    \Lambda := \Big \{x_{\frac{k}{2},j} \in \R : x_{\frac{k}{2},j} = \frac{k}{2} + \eta_{\frac{k}{2},j}, \text{ where } k \in \Z,  \text{ }  j= 1, \dots ,j_{\max}(\tfrac{k}{2}), \text{ } \eta_{\frac{k}{2},j} \in [0,1/2)\Big\}\, .
\end{equation}
We always assume that the sampling points $x_{k/2,j}$ are consecutive points in  $[k/2,k/2+1/2)$, i.e., $x_{k/2,j} < x_{k/2,j+1}$ for $j= 1, \dots ,   j_{\max}(\tfrac{k}{2})-1$. We measure the local sampling density by the maximum gap between the samples, i.e. 
\begin{equation}
    \delta_\frac{k}{2} =
    \max_{j=1,\dots,j_{\max}(\frac{k}{2})-1}(x_{\frac{k}{2},j+1}-x_{\frac{k}{2},j}), 
\end{equation}
where the number $j_{\max}(\frac{k}{2})$ depends on the interval.

In the following we will represent the functions in $\pw $ in the form  
\begin{equation} \label{eq:c2}
    f(x) 
    = \sum_{k \in \Z} \sum_{l = 0}^{b(k)}c_{k,l}\psi_{k,l}(x)
    = \sum_{k \in \Z} P_k(x) g(x-k/2),
\end{equation}
where $P_k(x):=\sum_{l = 0}^{b(k)} c_{k,l}d_l(M_l + (-1)^{l+k}M_{-l})$ is a trigonometric polynomial of degree $b(k)$ and period 1. Roughly speaking, a function in $\pw $ restricted to $[\tfrac{k}{2},\tfrac{k+1}{2}]$ behaves like a trigonometric
polynomial of degree $b(k)$.  We will use the fact that in this representation of $f\in \pw $ 
$$
\sum _{l=0}^{b(k)} |c_{k,l}|^2 = \|P_k\|_2^2 = \int _c^{c+1} |P_k(x)|^2 \, dx 
$$
(for every $c\in \R $) and
$$
\sum _{k\in \Z } \|P_k\|_2^2 = \|f\|_2^2 \, . 
$$

Since according to our model space $PW_b^2(g,\R)$, the local bandwidth on $[k/2,k/2+1/2)$ is roughly the degree $b(k)$ of the local polynomial $P_k$, 
we expect that the gap $\delta_{k/2}$ on $[k/2,k/2+1/2)$ is related to $b(k)$. 

The following theorem makes this expectation rigorous and provides sufficient condition for sampling for $PW_b^2(g,\R)$. For the proof we use smoothness estimates in $\pw$ (via Wirtinger's and Bernstein's inequality). For bandlimited functions this method goes back to \cite{gr92}, and was later called the adaptive weights method in \cite{fg94}. An axiomatic, extremely general approach was later presented in~\cite{fm11}.

\begin{theorem}\label{thm:suffcond}
    Let $g \in \mathcal{C}^1(\R)$ be a real-valued, even function with $\supp(g) \subseteq [-m, m]$ that generates an \owb .
    Let $PW_b^2(g,\R)$ be the space of variable bandwidth defined in \eqref{def:PW} and let $\Lambda \subseteq \R$ be as in
    \eqref{Lambda_ss} with $x_{\frac{k}{2},1}-\tfrac{k}{2} \leq \tfrac{1}{2}\delta_{\frac{k}{2}}$ and $\frac{k+1}{2}-x_{\frac{k}{2},j_{\max}(\frac{k}{2})} \leq \tfrac{1}{2}\delta_{\frac{k}{2}}$. 
    Define 
    \begin{equation}\label{def:D}
        D \coloneqq (4m) \cdot  \max \{2 \pi \norm{g}_\infty, \norm{g'}_\infty\}
    \end{equation}
    and 
    \begin{equation}\label{def:mu}
        \mu_{\frac{k}{2}} \coloneqq \max_{n =(k-2m,k+2m+1)\cap \Z} \Big(b(n)+1\Big).
    \end{equation}
    If for every $ k \in \Z$
    \begin{equation}\label{eq:SuffCond}
        \d_{\frac{k}{2}} < \frac{\pi}{\mu_{\frac{k}{2}} D},
    \end{equation}
    then $\Lambda $ is a sampling set and satisfies a weighted sampling inequality
    \begin{equation}\label{eq:c1}
        A\norm{f}_2^2 \leq  \sum_{k \in \Z} \sum_{j = 1}^{j_{\max}(\frac{k}{2})}|f(x_{\frac{k}{2},j})|^2
        w_{\frac{k}{2},j}  \leq B\norm{f}_2^2 \qquad \text{ for all } f\in \pw \, ,
    \end{equation}
    where $w_{\frac{k}{2},j} >0$ is a suitable sequence of weights. As a consequence, every $f \in PW_b^2(g,\R)$ can be reconstructed completely from the samples in $\Lambda$. 
\end{theorem}

\begin{remark}
    We have chosen a formulation that separates the influence of the window $g$, the local bandwidth $\mu _{k/2}$, and the local sampling density $\delta _{k/2}$ into different constants. In this way, the sufficient
    condition~\eqref{eq:SuffCond} looks formally similar to the maximum gap condition for sampling of bandlimited functions, e.g., in~\cite{gr92}. 
  \end{remark}

\begin{proof} 
    Let $f \in PW_b^2(g,\R)$. We define $y_{\frac{k}{2},0} = k/2$ and $y_{\frac{k}{2},j_{\max}(\frac{k}{2})} = k/2+1/2$, and let $y_{\frac{k}{2},j}=\frac{1}{2}(x_{\frac{k}{2},j}+x_{\frac{k}{2},j+1})$, $j= 1, \dots, j_{\max}(\frac{k}{2}) -1$ be the midpoints between the samples. Set $\indicator_{\frac{k}{2},j} = \indicator_{[y_{\frac{k}{2},j-1},y_{\frac{k}{2},j})}$. Moreover, if the maximum gap on $[\tfrac{k}{2}, \tfrac{k+1}{2}]$ is $\delta_\frac{k}{2}$, then $y_{\frac{k}{2},j}-x_{\frac{k}{2},j} \leq \frac{1}{2}\delta_\frac{k}{2}$ and $x_{\frac{k}{2},j}-y_{\frac{k}{2},j-1}\leq \frac{1}{2}\delta_{\frac{k}{2}}$.

    \vspace{2mm}
    
    \textbf{Step 1. Approximation from samples.} Let $P$ be the orthogonal projection from $\Lt$ onto $PW_b^2(g,\R)$.  Motivated by \cite{gr92}, we approximate $f$ by a step function and project onto $PW_b^2(g,\R)$ with $P$. The resulting approximation operator is defined by
    \begin{equation}
        Af = P \Big (\sum_{k \in \Z} \sum_{j = 1}^{j_{\max}(\frac{k}{2})} f\bigl(x_{\frac{k}{2},j}\bigr) \indicator_{\frac{k}{2},j}\Big).
    \end{equation}
    If $f\in \pw$, then $Af \in \pw$ and $A$ maps $\pw $ to $\pw $. We emphasize that $A$ uses only the samples $f(x_{\frac{k}{2},j})$. The question is how well the operator $A$ approximates the identity.
    
    Since $f = Pf = P(\sum_{k,j}f\indicator_{\frac{k}{2},j})$ and the characteristic functions $\indicator_{\frac{k}{2},j}$ have mutually disjoint support, we have
     \begin{align}
        \norm{f-Af}_2^2 & = 
        \Bigl\lVert P \Bigl (\sum_{k \in \Z} \sum_{j = 1}^{j_{\max}(\frac{k}{2})}(f - f(x_{\frac{k}{2},j})) \indicator_{\frac{k}{2},j}\Bigr) \Bigr\rVert_2^2 \\
        &\leq \Bigl\lVert \sum_{k \in \Z} \sum_{j = 1}^{j_{\max}(\frac{k}{2})}(f - f(x_{\frac{k}{2},j})) \indicator_{\frac{k}{2},j}\Bigr \rVert_2^2 \label{ineq:diff}\\
        & = \int_\R \Bigl \lvert\sum_{k \in \Z} \sum_{j = 1}^{j_{\max}(\frac{k}{2})}(f(x) - f(x_{\frac{k}{2},j})) \indicator_{\frac{k}{2},j}(x)\Bigr \rvert^2dx\\
        & = \int_\R \sum_{k \in \Z} \sum_{j = 1}^{j_{\max}(\frac{k}{2})}\Bigl\lvert(f(x) - f(x_{\frac{k}{2},j})) \indicator_{\frac{k}{2},j}(x)\Bigr\rvert^2dx .
    \end{align}
    The sums converge absolutely, hence we can interchange summation and integration. Thus we can write
    \begin{equation*}
        \norm{f-Af}_2^2 \leq \sum_{k\in \Z}\sum_{j = 1}^{j_{\max}(\frac{k}{2})} \int_{y_{\frac{k}{2},j-1}}^{y_{\frac{k}{2},j}}|f(x)-f(x_{\frac{k}{2},j})|^2dx.
    \end{equation*}
    By applying Wirtinger's inequality \eqref{WirtIneq} to every term, we obtain
    \begin{align*}
        \sum_{k\in \Z}\sum_{j = 1}^{j_{\max}(\frac{k}{2})} & \int_{y_{\frac{k}{2},j-1}}^{y_{\frac{k}{2},j}}|f(x)-f(x_{\frac{k}{2},j})|^2dx\\
        & \leq \frac{4}{\pi ^2} \sum_{k\in \Z}\sum_{j = 1}^{j_{\max}(\frac{k}{2})} \max \Bigl\{(x_{\frac{k}{2},j}-y_{\frac{k}{2},j-1})^2,(y_{\frac{k}{2},j}-x_{\frac{k}{2},j})^2\Bigr\} \int_{y_{\frac{k}{2},j-1}}^{y_{\frac{k}{2},j}}|f'(x)|^2dx\\
        & \leq \sum_{k\in \Z} \frac{1}{\pi ^2}\delta_\frac{k}{2}^2 \sum_{j = 1}^{j_{\max}(\frac{k}{2})}  \int_{y_{\frac{k}{2},j-1}}^{y_{\frac{k}{2},j}}|f'(x)|^2dx.
    \end{align*}
    By fixing $k \in \Z$, we get for a single term
    \begin{equation*}
        \frac{1}{\pi ^2}\delta_\frac{k}{2}^2 \sum_{j = 1}^{j_{\max}(\frac{k}{2})}  \int_{y_{\frac{k}{2},j-1}}^{y_{\frac{k}{2},j}}|f'(x)|^2dx 
        = \frac{1}{\pi ^2}\delta_\frac{k}{2}^2 \int_{\frac{k}{2}}^{\frac{k}{2}+\frac{1}{2}}|f'(x)|^2dx = \frac{1}{\pi ^2}\delta_\frac{k}{2}^2 \norm{f'\indicator_{[\frac{k}{2},\frac{k}{2}+\frac{1}{2})}}_2^2.
    \end{equation*}
    
    \vspace{2mm}
    
    \textbf{Step 2. Estimate of local smoothness}
    $\int_{\frac{k}{2}}^{\frac{k}{2}+\frac{1}{2}}|f'(x)|^2dx$.
    
    As in~\eqref{eq:c2} we write $f$ in the form $f(x) = \sum_{n \in \Z} P_n(x) g(x-n/2) $ with 
    $ P_n(x):=\sum_{l =0}^{b(n)} c_{n,l}d_l(M_l + (-1)^{l+n}M_{-l})$ and $\sum _{n\in \Z} \|P_n\|_2^2 = \|f\|_2^2$.  
    Then the  derivative $f'$ is
    \begin{equation}
        f'(x) = \sum_{n \in \Z} 
            \Big ( P_n'(x) T_\frac{n}{2}g(x) + P_n(x) T_\frac{n}{2}g'(x)
            \Big ).
    \end{equation}
    Since $\supp(g) \subseteq [-m,m]$, the sum is locally finite. The assumption $x \in \bigl[\frac{k}{2}, \frac{k}{2} + \frac{1}{2}\bigr)$ requires  $x - \frac{n}{2} \in (-m,m)$. This restricts the summation to indices 
    $n \in (k-2m,k + 1 + 2m) \cap \Z$. We abbreviate this interval by  
    \begin{equation}\label{def:Fk}
        F_k \coloneqq (k-2m,k + 1 + 2m) \cap \Z.
    \end{equation} 
    Then we have
    \begin{align}\label{deriv_terms}
         \int_{\frac{k}{2}}^{\frac{k}{2}+\frac{1}{2}}|f'(x)|^2dx &= 
        \int_{\frac{k}{2}}^{\frac{k}{2}+\frac{1}{2}} \Big |\sum_{n \in F_k}\Big ( P_n'(x) T_\frac{n}{2}g(x) + P_n(x) T_\frac{n}{2}g'(x) \Big )
        \Big |^2dx \nonumber \\
        & = \sum_{n,j \in F_k}\int_{\frac{k}{2}}^{\frac{k}{2}+\frac{1}{2}} \Big ( 
            P_n'(x) T_\frac{n}{2}g(x)\overline{P_j'(x) T_\frac{j}{2}g(x)}\\
        & \qquad \qquad \qquad    
        + 2 \Rp \Big\{P_n'(x) T_\frac{n}{2}g(x) \overline{P_j(x) T_\frac{j}{2}g'(x)} \Big\} \nonumber \\
        & \qquad \qquad \qquad + P_n(x) T_\frac{n}{2}g'(x)\overline{P_j(x) T_\frac{j}{2}g'(x)} \Big )dx \nonumber.
    \end{align}
    To bound the first term, we apply Bernstein's inequality \eqref{eq:BernIneq} to the polynomial $P_n$ of degree $b(n)$ and use the obvious estimate 
    $\norm{T_\frac{n}{2}gT_\frac{j}{2}\overline{g}}_\infty \leq
    \norm{g}_\infty^2$. It follows that 
    \begin{align}
        \int_{\frac{k}{2}}^{\frac{k}{2}+\frac{1}{2}} \Big ( 
            P_n'(x) T_\frac{n}{2}g(x)\overline{P_j'(x) T_\frac{j}{2}g(x)} \Big )dx 
        & \leq \norm{T_\frac{n}{2}gT_\frac{j}{2}\overline{g}}_\infty\norm{P_n'}_2\norm{P_j'}_2\\
        & \leq \norm{g}_\infty^2 4 \pi^2  b(n) b(j) \norm{P_n}_2\norm{P_j}_2.
    \end{align}
    Taking the sum over $n$ and $j$, we obtain
    \begin{align}\label{ineq:Int1}
        \sum_{n,j \in F_k} &\int_{\frac{k}{2}}^{\frac{k}{2}+\frac{1}{2}} \Big ( 
            P_n'(x) T_\frac{n}{2}g(x)\overline{P_j'(x) T_\frac{j}{2}g(x)}\Big )dx\\
        & \leq 4 \pi^2 \norm{g}_\infty^2   \sum_{n,j \in F_k} b(n) b(j) \norm{P_n}_2\norm{P_j}_2\\
        & = 4 \pi^2 \norm{g}_\infty^2  \Big ( \sum_{n \in F_k} b(n) \norm{P_n}_2\Big )^2.
    \end{align}
    For the middle terms of \eqref{deriv_terms} we obtain in a similar way that
    \begin{align}
        2 &\int_{\frac{k}{2}}^{\frac{k}{2}+\frac{1}{2}} \Rp \Big\{P_n'(x) T_\frac{n}{2}g(x) \overline{P_j(x) T_\frac{j}{2}g'(x)} \Big\}dx\\ 
        & \leq 2\norm{T_\frac{n}{2}gT_\frac{j}{2}\overline{g'}}_\infty \norm{P_n'}_2\norm{P_j}_2\\
        & \leq 4 \pi \norm{g}_\infty \norm{g'}_\infty b(n) \norm{P_n}_2\norm{P_j}_2.
    \end{align}
    After summing over $n$ and $j$, we have
    \begin{align}\label{ineq:Int2}
        \sum_{n,j \in F_k} & \int_{\frac{k}{2}}^{\frac{k}{2}+\frac{1}{2}} 2 \Rp \Big\{P_n'(x) T_\frac{n}{2}g(x) \overline{P_j(x) T_\frac{j}{2}g'(x)} \Big\}dx\\
        & \leq  4 \pi \norm{g}_\infty \norm{g'}_\infty \Big ( \sum_{n \in F_k} b(n) \norm{P_n}_2\Big )\Big ( \sum_{n \in F_k} \norm{P_n}_2\Big ).
    \end{align}
    The bound for the last term of \eqref{deriv_terms} is
    \begin{align}\label{ineq:Int3}
        \sum_{n,j \in F_k} & \int_{\frac{k}{2}}^{\frac{k}{2}+\frac{1}{2}} \Big (P_n(x) T_\frac{n}{2}g'(x)\overline{P_j(x) T_\frac{j}{2}g'(x)} \Big )dx\\
        & \leq \sum_{n,j \in F_k} \norm{T_\frac{n}{2}g'T_\frac{j}{2}\overline{g'}}_\infty \norm{P_n}_2 \norm{P_j}_2\\
        & \leq \norm{g'}^2_\infty \Big ( \sum_{n \in F_k} \norm{P_n}_2 \Big)^2.
    \end{align}

    \vspace{2mm}
    
    \textbf{Step 3. Intermediate estimate for $f-Af$. } Adding all terms, the local estimate for the derivative is
    \begin{align*}
        \int_{\frac{k}{2}}^{\frac{k}{2}+\frac{1}{2}}|f'(x)|^2dx & \leq
        4 \pi^2 \norm{g}_\infty^2 \Big ( \sum_{n \in F_k} b(n) \norm{P_n}_2 \Big )^2 + \norm{g'}^2_\infty \Big ( \sum_{n \in F_k} \norm{P_n}_2 \Big)^2\\
        & \qquad + 4 \pi \norm{g}_\infty \norm{g'}_\infty \Big ( \sum_{n \in F_k} b(n) \norm{P_n}_2\Big )\Big ( \sum_{n \in F_k} \norm{P_n}_2\Big )\\
        & = \Big ( 2 \pi \norm{g}_\infty \sum_{n \in F_k} b(n) \norm{P_n}_2  + \norm{g'}_\infty \sum_{n \in F_k} \norm{P_n}_2 \Big )^2\\
        & \leq \max \{2 \pi \norm{g}_\infty, \norm{g'}_\infty\}^2 \Big (  \sum_{n \in F_k} \norm{P_n}_2 \bigr(b(n)+1 \bigl) \Big )^2.
    \end{align*}   
    Summing all the terms, we obtain
    \begin{align}
        \norm{f-Af}_2^2  
        &\leq \sum_{k \in \Z} \frac{1}{\pi ^2} \d_\frac{k}{2}^2 \int_{\frac{k}{2}}^{\frac{k}{2}+\frac{1}{2}}|f'(x)|^2dx\\
        &\leq \frac{1}{\pi ^2}\max \{2 \pi \norm{g}_\infty, \norm{g'}_\infty\}^2 \sum_{k \in \Z} \d_\frac{k}{2}^2 \Big ( \sum_{n \in F_k} \norm{P_n}_2 \bigr(b(n)+1 \bigl)  \Big )^2.
    \end{align}
    We define the active local bandwidth as 
    \begin{equation}
        \mu_{\frac{k}{2}} \coloneqq \max_{n =(k-2m,k+2m+1)\cap \Z}
        \Big(b(n)+1\Big) \, .
    \end{equation}
    Then  we can bound $f-Af$  in the following way
    \begin{equation}
        \norm{f-Af}_2^2 \leq \frac{1}{\pi ^2} \max \{2 \pi \norm{g}_\infty, \norm{g'}_\infty\}^2 \sum_{k \in \Z} \d_\frac{k}{2}^2 \mu_{\frac{k}{2}}^2 \Big ( \sum_{n \in F_k} \norm{P_n}_2 \Big )^2.
    \end{equation}

    \vspace{2mm}
    
    \textbf{Step 4. Final estimate for $f-Af$.}  
    We interpret $\sum_{n\in F_k} \norm{P_n}_2$ as a convolution. Let $\alpha :  \Z  \to \R$ be the sequence $\alpha(n) =\norm{P_n}_2$ and $\indicator_{(-2m-1,2m)}$ be the characteristic function of $(-2m-1,2m)$. 
    Then
    \begin{align}
        \Big ( \alpha \ast \indicator_{(-2m-1,2m)}\Big) (k) 
        & = \sum_{n \in \Z} \alpha(n) \indicator_{(-2m-1,2m)} (k - n) 
        = \sum_{n \in \Z} \alpha(n) \indicator_{(k-2m,k+2m+1)} (n)\\
        & = \sum_{n \in F_k} \alpha(n) = \sum_{n  \in F_k}\norm{P_n}_2.
    \end{align}
    We can rewrite the error as
    \begin{equation}
        \norm{f-Af}_2^2 \leq \frac{1}{\pi ^2} \max \{2 \pi \norm{g}_\infty, \norm{g'}_\infty\}^2 \sum_{k \in \Z} \d_\frac{k}{2}^2 \mu_{\frac{k}{2}}^2 \Big ( \alpha \ast \indicator_{(-2m-1,2m)}(k) \Big )^2.
    \end{equation}
    The right-hand side can readily be estimated with Young's inequality as
    \begin{align*}
        \sum_{k \in \Z} \d_\frac{k}{2}^2 \mu_{\frac{k}{2}}^2\Big ( \alpha \ast \indicator_{(-2m-1,2m)}(k)\Big)^2
        &\leq \sup_{k \in \Z} \bigl(\d_\frac{k}{2}^2 \mu_{\frac{k}{2}}^2\bigr) \norm{\alpha \ast\indicator_{(-2m-1,2m)}}_2^2\\
        &\leq \sup_{k \in \Z} \bigl(\d_\frac{k}{2}^2 \mu_{\frac{k}{2}}^2\bigr)
        \norm{\indicator_{(-2m-1,2m)}}_1^2 \norm{\alpha}_2^2 \\ 
        &\leq \sup_{k \in \Z} \bigl(\d_\frac{k}{2}^2 \mu_{\frac{k}{2}}^2\bigr) (4m)^2 \sum_{k \in \Z}\norm{P_k}_2^2
    \end{align*}
    since $\norm{\indicator_{(-2m-1,2m)}}_1 = 4m$.
    Defining $D \coloneqq 4m\, \max \{2 \pi \norm{g}_\infty, \norm{g'}_\infty\}$, we rewrite
    \begin{equation}
        \norm{f-Af}_2^2
        \leq \frac{1}{\pi ^2} D^2\sup_{k \in \Z} \bigl(\d_\frac{k}{2}^2 \mu_{\frac{k}{2}}^2\bigr) \sum_{k \in \Z}\norm{P_k}_2^2.
    \end{equation}
    Hypothesis~\eqref{eq:SuffCond} says that   $\frac{1}{\pi ^2} D^2 \d_\frac{k}{2}^2 \mu_{\frac{k}{2}}^2 < 1$ for every $k \in \Z$,  therefore the error constant $\gamma^2 \coloneqq  \frac{1}{\pi ^2} D^2  \sup_{k \in \Z} \bigl(\d_\frac{k}{2}^2 \mu_{\frac{k}{2}}^2\bigr) $ satisfies $\gamma < 1$. Hence, 
    \begin{equation}\label{ineq:boundI-Af}
        \norm{f-Af}_2^2 
        \leq \gamma^2 \sum_{k \in \Z} \norm{P_k}^2_2
        = \gamma^2 \norm{f}_2^2.
    \end{equation}
    Consequently,  the operator $A$ is invertible on $PW_b^2(g,\R)$ with an inverse $A^{-1}$ satisfying 
    \begin{equation}
       \norm{A^{-1}}_{\text{op}} = 
       \Bigl \lVert\sum_{n = 0}^{\infty}(\text{Id}-A)^n \Bigr \rVert_{\text{op}} 
       \leq \sum_{n = 0}^{\infty} \lVert \text{Id}-A \rVert^n_{\text{op}}
       = \frac{1}{1-\gamma},
    \end{equation}
    
    \vspace{2mm}
    
    \textbf{Step 5. Reconstruction and sampling estimate.}
    It follows that $f=A^{-1} (Af)$ with $Af$ using only the samples on $\Lambda = \{x_{\frac{k}{2},j} \}$.  Hence $f$ is completely determined by $f(x_{\frac{k}{2},j})$ for every $k \in \Z$ and $j=1,\dots,j_{\max}(\frac{k}{2})$.

    The lower sampling inequality in~\eqref{eq:c1}, equivalently, the stability of the reconstruction, follows from
    \begin{align}
        \norm{f}_2^2 &= \norm{A^{-1}Af}_2 ^2
        \leq \norm{A^{-1}}_{\text{op}}^2\norm{P}_{\text{op}}^2 \Bigl \lVert\sum_{k \in \Z} \sum_{j = 1}^{j_{\max}(\frac{k}{2})}f\bigl(x_{\frac{k}{2},j}\bigr)\indicator_{\frac{k}{2},j}\Bigr \rVert_{2}^2 \\
        &= \frac{1}{(1-\gamma)^2} \, \sum _{k\in \Z} \sum
        _{j=1}^{j_{\max}(\frac{k}{2})} |f(x_{\frac{k}{2},j})|^2
        w_{\frac{k}{2},j} \, ,
    \end{align}
    where the weights are given by  
    \begin{equation}\label{weights}
        w_{\frac{k}{2},j} =  \int \indicator_{\frac{k}{2},j}.
    \end{equation}
    To get the right-hand side of the inequality \eqref{eq:c1} we use the previous computations and we obtain
    \begin{align}
        \sum _{k\in \Z} \sum_{j=1}^{j_{\max}(\frac{k}{2})}|f(x_{\frac{k}{2},j})|^2 w_{\frac{k}{2},j}
        & = \Bigl \lVert \sum_{k \in \Z} \sum_{j = 1}^{j_{\max}(\frac{k}{2})}f\bigl(x_{\frac{k}{2},j}\bigr)\indicator_{\frac{k}{2},j}\Bigr \rVert_{2}^2\\
         &\leq \Bigl ( \norm{f}_2 + \Bigl \lVert \sum_{k \in \Z} \sum_{j = 1}^{j_{\max}(\frac{k}{2})}f\bigl(x_{\frac{k}{2},j}\bigr)\indicator_{\frac{k}{2},j} -f \Bigr \rVert_{2} \Bigl)^2.
    \end{align}
    By the inequalities \eqref{ineq:diff} and \eqref{ineq:boundI-Af}, we directly have
    \begin{equation}
         \sum _{k\in \Z} \sum_{j=1}^{j_{\max}(\frac{k}{2})}|f(x_{\frac{k}{2},j})|^2 w_{\frac{k}{2},j} 
         \leq (1 + \gamma)^2 \norm{f}_2^2
    \end{equation}
\end{proof}

\begin{remarks} { \rm
    (i) Note that the separation of the points is not a requirement for the proof. In fact, a local variation of the density can be addressed by using the adaptive weights method (sometimes called ``density compensation
    factors'' \cite{gs14}). 
    If the sampling points are separated, i.e, $\inf _{\lambda , \mu \in \Lambda , \lambda \neq \mu }|\lambda - \mu | = \nu >0$, then $w_{\frac{k}{2},j}\geq \nu >0$ and \eqref{eq:c1} becomes the standard sampling inequality
    \begin{equation}\label{eq:c1a}
        A'\norm{f}_2^2 \leq  \sum_{k \in \Z} \sum_{j = 1}^{j_{\max}(\frac{k}{2})}|f(x_{\frac{k}{2},j})|^2
        \leq B'\norm{f}_2^2 \qquad \text{ for all } f\in \pw \, .
    \end{equation}
    
    (ii) The same proof can be employed when substituting a Wilson orthonormal basis with a Wilson Riesz basis. In this case, we use the expansion $f= \sum _{n,l} \langle f, \tilde{\psi}_{n,l}\rangle \psi_{n,l}$,  where $\{\tilde{\psi}_{n,l} \}$ is the biorthogonal basis of $\{\psi_{n,l} \}$. In \eqref{ineq:boundI-Af} we then obtain $\sum _{n\in \Z} \|P_n\|_2^2 = \sum _{n,l} | \langle f, \tilde{\psi}_{n,l}\rangle|^2 \leq B \|f\|_2^2$ for $f\in \pw $ with the Riesz basis constant $B$. The rest of the proof is identical. 
    
    (iii) Except for very special cases, as in
    ~\cite{asr17,gr92,wa87}, the adaptive weights method does not
    produce optimal results and fails to come close to the necessary
    sampling density. To stick to the essentials, we have therefore
    not aimed to optimize (and complicate) the proof at all steps. The
    estimates can be slightly sharpened in several places, but this
    does not seem relevant in the numerical simulations. 
    
    (iv) In~\cite{ag23} we formulated a version that uses windows with
    $\supp (g) \subseteq [-1/2 - \epsilon , 1/2+\epsilon ]$ and found
    a different sufficient condition for sampling. See
    also~\cite{an24}.   } 
\end{remarks}

\begin{example} \label{ex:c1}
    Consider the window  $g(x) = \sqrt{2}\cos(\pi x)\indicator_{[-1/2,1/2]}(x)$, as plotted in Figure \ref{fig:CosWind}. Using the characterization (iii) of Theorem~\ref{WB_char} it is easy to see that the associated Wilson system $\{ \psi _{n,l}\}$ is an orthonormal basis for $L^2(\R)$.
    Although the function $g$  lacks differentiability at the points
    $\{-1/2, 1/2\}$, the conclusion of Theorem~\ref{thm:suffcond}
    remains applicable. The proof requires only that
    $g$ is 
    differentiable on  each interval $(k/2, k/2+1/2)$ with one-sided
    derivatives at $k/2$, $k\in \Z $.
    
    This Wilson basis reveals an important point of Theorem~\ref{thm:suffcond} and shows that the local quantity $\mu_{k/2} =  \max_{n =(k-2m,k+2m+1)\cap \Z} \big(b(n)+1\big)$ must occur in the sufficient condition. 
    
    Since $\supp(g) = (-1/2, 1/2)$, only two translates of $g$ overlap. Thus the restriction of a general function $f=\sum_n P_n(x) g(x-n/2)$ to  $[0,1/2]$ (or any other interval $[k/2,k/2+1/2]$) is
    \begin{align*}
        f\Big| _{[0,\frac{1}{2}]} &= P_0(x) g(x)+  P_{1}(x)g\Big(x-\frac{1}{2}\Big)  \\
        &=P_0(x) \cos \pi x + P_1(x) \sin \pi x  \\
        &=\frac{e^{-i\pi x}}{2} \Big( P_0(x) (e^{2\pi i x} + 1) -i P_1(x) (e^{2\pi i x} - 1)\Big)
    \end{align*}
    If $\mathrm{deg}\, P_0 = b(0)$ and $\mathrm{deg}\, P_1 = b(1)$, then $e^{-i\pi x} f\Big| _{[0,\frac{1}{2}]}$ is a trigonometric polynomial of degree $\max \{b(0), b(1)\} +1$, which is precisely definition~\eqref{def:mu}. The formulation of Theorem~\ref{thm:suffcond} is a far-reaching generalization of this observation and highlights the importance of the  value $\mu_{k/2}$ as a more accurate measure of local bandwidth compared to $b(k)$.
\end{example}

\section{Numerical simulations}\label{Sec5}
In this section, we present various sampling  scenarios and the results of numerical reconstructions. The main point is the numerical and pictorial illustration of our notion of variable bandwidth. We hope to convince the reader that the spaces $\pw $ are meaningful and useful notions to capture variable bandwidth. 

In principle the proof of Theorem~\ref{thm:suffcond} provides a reconstruction procedure as follows: from the samples $\{f(x_{\frac{k}{2},j})\}$ we form the approximation operator $Af$ and then invert $A$ to obtain the reconstruction $f= A^{-1} (Af)$. Since this reconstruction requires the orthogonal projection $P$ onto $\pw$ and the inversion of $A$, this method is not practical.

A second reconstruction method can be based on the sampling inequality~\eqref{eq:c1}. Since $\pw$ is a reproducing kernel Hilbert space, \eqref{eq:c1} can be read as the frame inequality for the reproducing kernels associated to the sampling points. Then any off-the-shelf frame algorithm~\cite{ch03} yields a reconstruction.  

In both cases one still has to deal with a numerical discretization from the infinite dimensional space $\pw$ (requiring infinitely many data) to some finite-dimensional problem that uses only finite information as input. One of the advantages of our model of variable bandwidth is that such a discretization is not necessary! Precisely, if the window $g$ generating the Wilson basis has compact support and if the samples are taken from an interval $I$, then we may try to find a reconstruction or approximation in the finite-dimensional space $\pwi $. In this case, a reconstruction amounts to the solution of a finite-dimensional linear system and can be treated with numerical linear algebra.  

\subsection{Matrix formulation of the reconstruction.}

Assume that $g$ has compact support in $[-m,m]$. Let $I = [\alpha,\beta]$ and recall that
$$
\pwi =  \spn\{\psi_{n,l} : \supp(\psi _{n,l}) \subseteq I\}
$$
is the local version of the variable bandwidth space $\pw$. 
According to the proof of Theorem~\ref{NCW:compact}, $\pwi $ is a finite-dimensional subspace  of $PW_b^2(g, \R)$ consisting of those functions $f\in \pw $ with $\supp(f) \subseteq I$.

Let $\Lambda = \{x_j: j=1, \dots , L\}$ be a sampling set contained in $I$. 

To reconstruct a function $f \in PW_b^2(g,I)$ from its samples $f(x_j)$, we need to solve the linear system of equations 
\begin{equation}\label{pwi_system}
    \sum_{\supp(\psi _{n,l})\subseteq I} c_{n,l}\psi_{n,l}(x_j) =
    \sum_{n \in [\alpha+m,\beta-m]\cap \Z}c_{n,0}\psi_{n,0}(x_j) +
    \sum_{n/2 \in [\alpha+m,\beta-m]\cap \Z}\sum_{l =
    1}^{b(n)}c_{n,l}\psi_{n,l}(x_j) = f(x_j) \, ,
\end{equation}
$j=1, \dots , L $,  for the coefficients $c_{n,l}$. 

This linear system can be written in matrix notation as
\begin{equation}
    Uc = y,
\end{equation}
where $y$ is the input  sequence $\{f(x_j)\}$, and the matrix $U$ has the entries 
\begin{equation}\label{eq:matrixU}
    U_{j,(n,l)} = \psi_{n,l}(x_j).
\end{equation}

In general the samples are not exact and may contain measurement errors or noise, so that the input vector $y$ is given by entries  $y_j = f(x_j) + \epsilon _j$. It may also happen that the sampled function is not contained in $\pwi$. In both cases we seek a function $\tilde{f}\in \pwi $ that approximates the data vector $y$ optimally. In the absence of additional information, such as  sparsity, we use a least squares approach and seek $\tilde{f}$ as the solution of
\begin{align}
    \tilde{f} 
    & = \mathrm{argmin} _{f\in \pwi } \sum _{j} |\tilde{f}(x_j) -  y_j|^2 \notag \\
    & = \mathrm{argmin} _{c_{n,l} } \sum _{j} \big|\sum _{n,l: \supp(\psi_{n,l})\subseteq I} c_{n,l} \psi_{n,l} (x_j) -  y_j\big|^2 \\
    & = \mathrm{argmin} _{c_{n,l} } \|Uc-y\|_2^2 \, .  \label{eq:c3}
\end{align}

The solution to the matrix least squares problem uses the Moore-Penrose pseudo-inverse matrix $U^\dagger$ of  $U$ and is given by 
\begin{equation}\label{pinv}
    c = U^\dagger  y.
\end{equation}

The approximating function $\tilde{f}\in \pwi $ is then
\begin{equation}\label{eq:c5}
  \tilde{f} = \sum _{\supp (\psi _{n,l})\subseteq I} c_{n,l} \psi _{n,l} \, 
\end{equation}
and solves \eqref{eq:c3}. 
If the data are exact, $y_j = f(x_j)$ for $f\in \pwi$, then
$\tilde{f} = f$ and the reconstruction is exact.

For a  detailed exposition of the least squares approximation method, see \cite{bj90, gvlc96, lhr74}.

\subsection{Numerical set-up.} \label{sec:51} For the numerical simulations we adopt \eqref{pinv} and \eqref{eq:c5} as our starting point. Next we explain the numerics and the selection of parameters in the experiments. 

\vspace{2mm}

\textbf{Window function.} For the numerical simulations we use the
window function   \\ $g(x) = \sqrt{2}\cos(\pi x)\indicator_{[-1/2,1/2]}(x)$ from Example~\ref{ex:c1} and Figure \ref{fig:CosWind}. Then $g$ generates an orthonormal Wilson basis. 
In this case the constants of Theorem~\ref{thm:suffcond} are readily computed and are listed in Table \ref{tab:PropCosWind}. 

\vspace{2mm}

\begin{minipage}{\textwidth}
    \begin{minipage}[c]{0.49\textwidth}
        \centering
        \includegraphics[width=70mm]{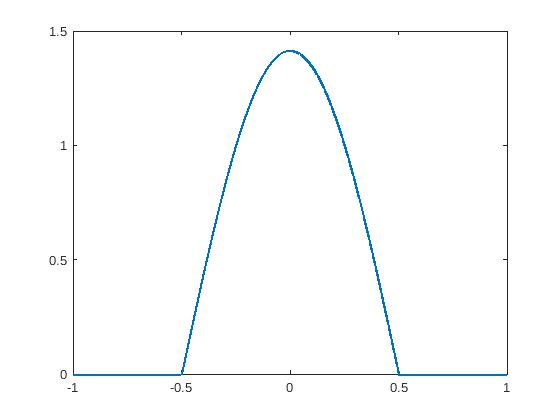}
        \captionof{figure}{Window\\
            $\sqrt{2}\cos(\pi x)\indicator_{[-1/2,1/2]}(x)$ used to generate $\pw $.}
         \label{fig:CosWind}
    \end{minipage}
    \hfill
    \begin{minipage}[c]{0.49\textwidth}
        \centering
        \begin{tabular}[b]{cc}\hline
            \toprule
            Window & $\sqrt{2}\cos(\pi x)\indicator_{[-1/2,1/2]}(x)$ \\
            \midrule
            $m$ & 0.5\\
            $\norm{g}_\infty$ &  $\sqrt{2}$\\
            $\norm{g'}_\infty$ & $\sqrt{2} \pi$\\ 
            \bottomrule
        \end{tabular}
        \captionof{table}{Constants of the window needed in Theorem \ref{thm:suffcond}.}
        \label{tab:PropCosWind}
    \end{minipage}
\end{minipage}

\bigskip

\textbf{Numerical error.} 
In the numerical simulations we start with a known function $f \in \pwi $ and then compute a reconstruction $\tilde{f}$ from its samples. To assess the quality of the reconstruction, we evaluate the discretized error on a uniform grid $\{\gamma n :n \in \Z\}$ on the interval of reconstruction $I$. 
We use the relative errors in the $\ell ^2$-norm and in the $\ell ^\infty $-norm defined by  
\begin{align}
    & E(f,2) \coloneqq \frac{\Big \{ \sum_{\gamma n \in I} |\tilde{f}(\gamma n)-f(\gamma n)|^2\Big \}^{1/2}}{\Big \{ \sum_{\gamma n \in I} |f(\gamma n)|^2\Big \}^{1/2}},\\
    & E(f,\infty) \coloneqq \frac{\max_{\gamma n \in I} |\tilde{f}(\gamma n)-f(\gamma n)|}{\max_{\gamma n \in I} |f(\gamma n)|},
\end{align}
where the step-size $\gamma$ is chosen to be of the order $\mathcal{O}(10^{-4})$.

\vspace{2mm}

\textbf{Least squares approximation.} The numerical approximation is
carried out with MATLAB, and the solution to \eqref{eq:c3} is obtained with
the MATLAB function  \verb|pinv|. Since our aim is to validate the
notion of variable bandwidth with Wilson basis, we have not tried to
optimize the numerical part. To set up the sampling matrix
\eqref{eq:matrixU} we choose a suitable enumeration of the indices 
$(n,l)$ satisfying  $\supp(\psi _{n,l}) \subseteq I$, namely,  
$$i = \sum_{\substack{r < n : \\ \supp(\psi_{r,l})\subseteq I}}b(r) + \#\{r\leq n: \supp(\psi_{r,0}) \subseteq I \} + l \, .$$
 The entries of \eqref{eq:matrixU} are identified by
$$ U_{j, (n,l)} = U_{j,i}.$$
As always,  the generation of $U$ is the most expensive part of the
computation. 

\vspace{2mm}

\textbf{Test signal.}
We sample a signal on the interval $I = [0,6]$, and we 
choose a suitable test signal $f_1 \in PW_b^2(g,I)$ by generating its coefficients.  
To begin, we generate a bandwidth sequence $b : \mathbb{Z} \to \mathbb{N}$, where $b(n) \sim \mathcal{U}_{[20,500]}$ for $\{n \in \Z : \supp(\psi_{n,l}) \subseteq I\}$, with $\mathcal{U}$ denoting the uniform distribution of integers in the range $[20,500]$. The resulting bandwidth sequence is $b = (102, 35, 499, 444, 341, 111, 197, 241, 492, 95, 431)$, illustrated by the yellow bands in Figure \ref{fig:Case1_spec}.
For clarity, we assume that only the maximum frequency on each interval $[k/2,k/2+1/2]$ is active, so that the test signal is of the form~\footnote{Note that this is a sparse signal that could be reconstructed with compressed sensing methods. Our choice $f_1$ is motivated by the need to obtain a reasonable and interpretable plot (Figure~\ref{fig:Case1_orig}). General signals in $\pwi$ are \emph{not} sparse, and the least square solution is the only option.}
\begin{equation}\label{eq:c7}
    f_1 = \sum_{n=1}^{11} c_{n,b(n)}\psi_{n,b(n)} \in \pwi  \, .  
\end{equation}
We randomize the coefficients by setting for $\{n \in \Z : \supp(\psi_{n,l}) \subseteq I\}$

\begin{equation}\label{def:GenerateCoeff}
    c_{n,b(n)}\sim \mathcal{U}_{[-200,200]} \cdot 10^{-2} \text{ and all the other coefficients are set to be zero.}
\end{equation}

By \eqref{eq:dim_pwi}, the dimension of $\pwi $ is given by 
\begin{align}\label{dim_case1}
    \dim(\pwi ) & = \dim (\spn \{\psi_{n,l} : \supp(\psi_{n,l}) \subseteq I\})\\
    & = \#([0.5,5.5] \cap \Z) + \sum_{n/2 \in [0.5,5.5]} b(n) = 5 + \sum_{n=1}^{11} b(n)= 2993.
\end{align} 

Figure \ref{fig:Case1_OS_banal} represents the generated signal $f_1 \in PW_b^2(g,I)$ that needs to be reconstructed. The local frequency spectrum is visualized by means of the spectrogram of $f_1$. 

\begin{figure}[h!]
    \centering
    \begin{subfigure}[b]{0.45\textwidth}
         \centering
         \includegraphics[width=70mm]{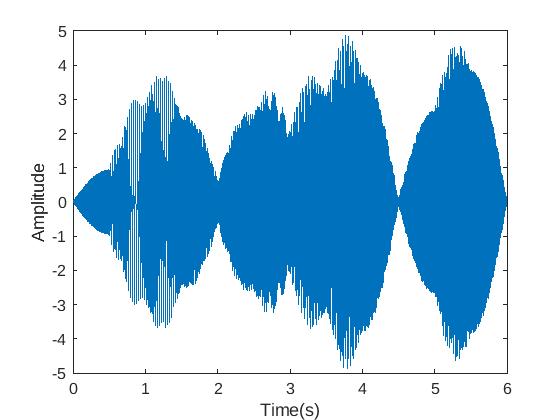}
         \caption{Function $f_1\in PW_b^2(g,I)$.}
         \label{fig:Case1_orig}
     \end{subfigure}
     \hfill
     \begin{subfigure}[b]{0.45\textwidth}
         \centering
         \includegraphics[width=70mm]{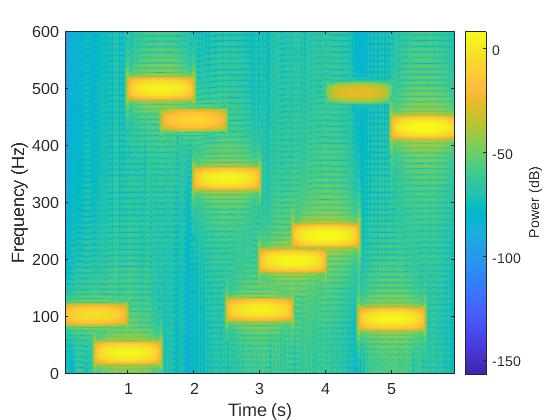}
         \caption{Spectrogram of $f_1$.}
         \label{fig:Case1_spec}
     \end{subfigure}
    \caption{Test signal $f_1\in PW_b^2(g,I)$ generated by the bandwidth sequence\\ 
    $b = (102,35,499,444,341,111,197,241,492,95,431)$, and coefficients sequence of the form \eqref{def:GenerateCoeff}.}
    \label{fig:Case1_OS_banal}
\end{figure}

\vspace{2mm}

\textbf{Sampling set.} For every signal we use two types of sampling sets. One set is chosen to satisfy the sufficient conditions of Theorem~\ref{thm:suffcond}, the other type is not covered by the theory.  

Since in our example the signal space $\pwi$ is assumed to be known, 
we can determine the local maximal gaps $\d_{k/2}$ on $I$  in advance by means of \eqref{eq:SuffCond}. 
We define 
\begin{equation}
    \mathcal{M} \coloneqq \Big\{k \in \Z: I \cap \Big[\frac{k}{2},\frac{k+1}{2}\Big) \neq \emptyset \Big\}
\end{equation}
to be the set of indices $k \in \Z$, for which we have to compute $\d_{k/2}$.
With $\mathcal{M} = \{0,\dots,11\}$, we choose the sampling set  
\begin{equation}\label{def:Lambda}
    \Lambda_{1} \coloneqq \Big\{x_{k/2,j} = \frac{k}{2} + \d_{k/2}j : \text{ } k \in \mathcal{M} , \text{ } j = 0,\dots,\Big\lfloor\dfrac{1}{2\d_{k/2}}\Big\rfloor\Big\}.
\end{equation}
The sampling points are equispaced according to the local bandwidth within the interval $[\frac{k}{2}, \frac{k}{2} + \frac{1}{2}]$ for every $k \in \mathcal{M}$, and by construction $\Lambda _1$ satisfies the conditions of Theorem~\ref{thm:suffcond}. 
We then evaluate the Wilson basis elements on  $\Lambda_{1}$ and
assemble them in the matrix $U$ according to \eqref{eq:matrixU} 
and then solve \eqref{pinv} and \eqref{eq:c5} by means of the MATLAB function \verb|pinv|.  

The reconstruction $\tilde{f}_1$ is accurate up to machine precision. Table \ref{tab:Err_Recons_cos_NoBoundEff_banal} reports the relative errors in the $\ell ^2$-norm and $\ell^\infty$-norm. The pointwise difference between the original function $f_1$ and its reconstruction $\tilde{f}_1$ is plotted in Figure \ref{fig:Case1_DiffORS_banal}. 

\vspace{2mm}

\begin{minipage}{\textwidth}
    \begin{minipage}[c]{0.49\textwidth}
        \centering
        \includegraphics[width=70mm]{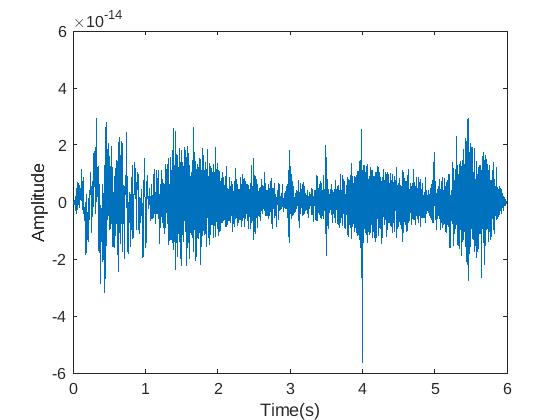}
        \captionof{figure}{Difference between the original signal $f_1$ and the reconstruction $\tilde{f}_1$. Maximal gap condition \eqref{eq:SuffCond} is satisfied.} 
        \label{fig:Case1_DiffORS_banal}
    \end{minipage}
    \hfill
    \begin{minipage}[c]{0.49\textwidth}
        \centering
        \begin{tabular}{cc} 
            \toprule
            \multicolumn{2}{c}{Relative errors}\\
            \midrule
            $E(f,2)$ & $3.776 \times 10^{-15}$\\
            $E(f,\infty) $ & $6.016 \times 10^{-15}$\\
            \bottomrule
        \end{tabular}
          \captionof{table}{Relative errors in $\ell^2$-norm and $\ell^\infty$-norm.}
          \label{tab:Err_Recons_cos_NoBoundEff_banal}
    \end{minipage}
\end{minipage}

\vspace{2mm}

However, the conditions of Theorem~\ref{thm:suffcond} are much too
pessimistic. According to Theorem~\ref{NCW:compact} we need at least
$\mathrm{dim}(\, \pwi)$ samples, whereas $\# \Lambda _1 =
12121$. Consequently, the redundancy, or oversampling factor, is given
by 
\begin{equation}
     q_{1} = \frac{\# \Lambda_1 }{\dim(\pwi)} = 4.05.
\end{equation}
It is therefore not surprising that with such a high oversampling the reconstruction works perfectly. 

For the sake of completeness, we conducted the same experiment with a function $f_2$ that shares the same bandwidth sequence as before but has full set of non-zero coefficients as follows: 

\begin{equation}
    c_{n,l}\sim \mathcal{U}_{[-200,200]} \cdot 10^{-2}, \text{ for } \{ n,l: \supp(\psi_{n,l}) \subseteq I\} 
\end{equation}

Consequently, the test signal can be expressed as
\begin{equation}\label{eq:f2}
    f_2 = \sum_{n=1}^{5} c_{n,0}\psi_{n,0} + \sum_{n=1}^{11} \sum_{l = 1}^{b(n)} c_{n,l}\psi_{n,l} \in \pwi  \, .  
  \end{equation}
  We use the same sampling set $\Lambda _1$ as before, so that the
  redundancy is again  $q_2 = q_1 = 4.05$. Then the
sufficient condition of Theorem~\ref{thm:suffcond} provides again a good
reconstruction, with relative errors in the $\ell^2$-norm and
$\ell^\infty$-norm of the order $\mathcal{O}(10^{-15})$. 
 The pointwise difference between the function $f_2$ and its reconstruction $\tilde{f}_2$ is plotted in Figure \ref{fig:Case1_DiffORS_Prec_NonZeroCoeff}.
Table \ref{tab:Err_Recons_NoBoundEff_NonZeroCoeff} reports the relative errors in the $\ell ^2$-norm and $\ell^\infty$-norm.

\begin{minipage}{\textwidth}
  \begin{minipage}[c]{0.49\textwidth}
    \centering
    \includegraphics[width=70mm]{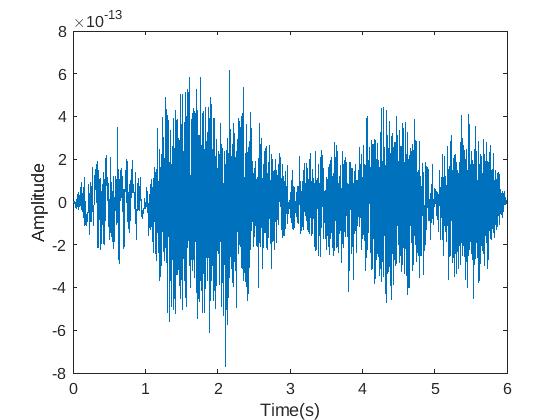}
    \captionof{figure}{Difference between the original signal $f_2$
      and the reconstruction $\tilde{f}_2$. 
      Maximal gap condition \eqref{eq:SuffCond} is satisfied.} 
    \label{fig:Case1_DiffORS_Prec_NonZeroCoeff}
  \end{minipage}
  \hfill
  \begin{minipage}[c]{0.49\textwidth}
    \centering
    \begin{tabular}{cc} 
        \toprule
        \multicolumn{2}{c}{Relative errors}\\
        \midrule
        $E(f,2)$ & $5.218 \times 10^{-15}$\\
        $E(f,\infty) $ & $4.956 \times 10^{-15}$\\
        \bottomrule
    \end{tabular}
      \captionof{table}{Relative errors in $\ell^2$-norm and $\ell^\infty$-norm.} 
      \label{tab:Err_Recons_NoBoundEff_NonZeroCoeff}
    \end{minipage}
\end{minipage}

\vspace{2mm}

In the following and more important experiment, we test sampling sets that are not covered by the theoretical conditions. Recall that $\mu_{k/2} = \max_{n =(k-2m,k+2m+1)\cap \Z} \big(b(n)+1\big)$ is the active local bandwidth. We now choose a sampling set of the form
\begin{equation} \label{def:Gamma_n}
    \Gamma_1^\rho \coloneqq \Big\{x_{k/2,j} = \frac{k}{2} + \frac{1}{2}\frac{1}{\rho\mu_{k/2}}j: \text{ } k \in \mathcal{M}, \text{ } j = 0,\dots,\floor{\rho\mu_{k/2}}\Big\}.
\end{equation}
with a parameter $\rho >0$. If $\rho$ is sufficiently large,  precisely $\rho\geq D/\pi = 5.66$, then $\Gamma_1^\rho $ is covered by Theorem \ref{thm:suffcond}. For $\rho < D/\pi$ we do not have any theoretical guarantees for the quality of the reconstruction.
The errors of the least square approximation are illustrated in Figure \ref{fig:Case1_Err} for different sampling sets $\Gamma_1^\rho$ with parameters $\rho = 0.7,0.8,\dots,2.4,2.5$. In all cases $\rho \geq 1$, machine precision is reached. The case $\rho = 1$ corresponds to an oversampling parameter of $1.43$ and the starting point $\rho = 0.7$ represents an oversampling parameter of $1.004$.

\begin{figure}
    \centering
    \begin{subfigure}[b]{0.45\textwidth}
         \centering
         \includegraphics[width=70mm]{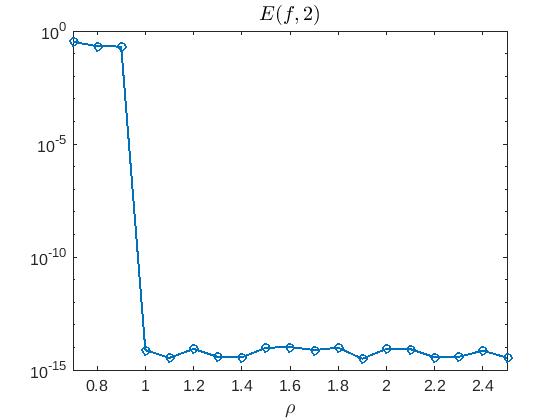}
         \caption{Relative error in the $\ell ^2$-norm.}
         \label{fig:Case1_Err2_banal}
     \end{subfigure}
     \hfill
     \begin{subfigure}[b]{0.45\textwidth}
         \centering
         \includegraphics[width=70mm]{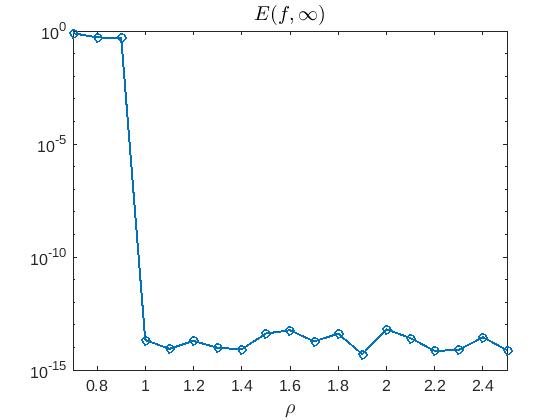}
         \caption{Relative error in the $\ell ^\infty$-norm.}
         \label{fig:Case1_ErrInfinity_banal}
     \end{subfigure}
    \caption{Behavior of the reconstruction of $f_1 \in PW_b^2(g,I)$ with increasing parameter $\rho = 0.7,0.8,\dots,2.4,2.5$.}
    \label{fig:Case1_Err}
\end{figure}

\subsection{Local reconstruction of a function of $PW_b^2(g,\R)$.}\label{Subsec5.2}
We now modify the scenario and assume that a function  $f_3\in \pw $ is given. We try to reconstruct or approximate $f_3$ on an interval $I$ from samples in $I$. This  situation is different, because a sample $x\in I$ sees all basis functions $\psi _{n,l}$ with $x \in \supp(\psi _{n,l})$. Consequently, we have to solve the linear equations
\begin{equation}\label{eq:c6}
    \sum _{(n,l): \supp(\psi_{n,l}) \cap I \neq \emptyset } c_{n,l} \psi_{n,l} (x_j) = y_j \qquad j=1, \dots, L \, .
\end{equation}
For $I=[0,6]$ and the window $g$ from Example~\ref{ex:c1} with $\supp(g)\subseteq [-1/2,1/2]$, the resulting equations are
\begin{equation}
    \sum _{n=0}^6 c_{n,0} \psi_{n,0} (x_j) + \sum _{n=0}^{12} \sum _{l=1}^{b(n)} c_{n,l} \psi_{n,l} (x_j) = y_j \qquad j=1, \dots, L \, .  
\end{equation}
The dimension of this problem (the number of variables) is now 
\begin{align}
    \dim(\spn\{\psi_{n,l} : \supp (\psi _{n,l}) \cap I \neq \emptyset\} )
    & = \# \{n \in \Z : \supp(\psi_{n,0}) \cap I\} + \sum_{\substack{n \in \Z :\\ \supp(\psi_{n,l})\cap I}}b(n)\\
    & = 7 + \sum_{n = 0}^{12}b(n)= 3101.
\end{align}

Figure \ref{fig:Case2_OS} illustrates the function $f_3\indicator_I$ that we aim to reconstruct.

\begin{figure}[h!]
    \centering
    \begin{subfigure}[b]{0.45\textwidth}
         \centering
         \includegraphics[width=70mm]{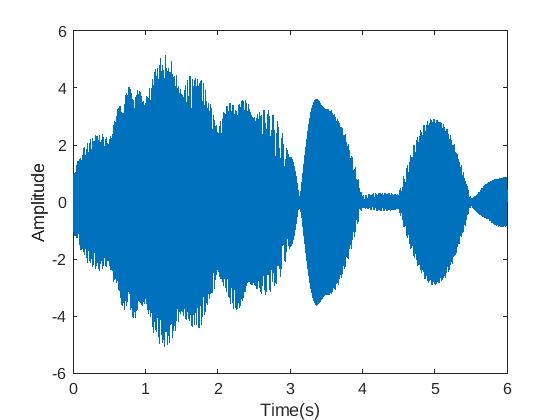}
         \caption{Function $f_3\indicator_I$.}
     \end{subfigure}
     \hfill
     \begin{subfigure}[b]{0.45\textwidth}
         \centering
         \includegraphics[width=70mm]{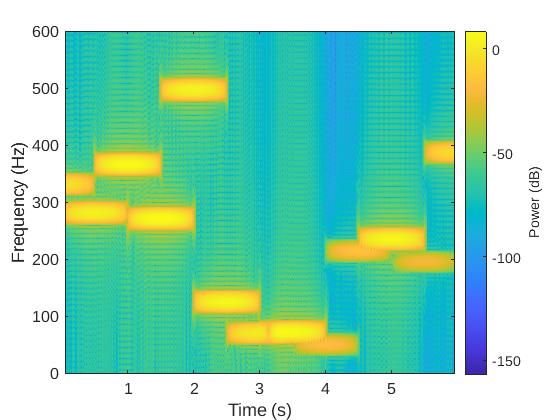}
         \caption{Spectrogram of $f_3\indicator_I$.}
     \end{subfigure}
    \caption{Signal $f_3\indicator_I$ generated by the bandwidth sequence\\
    $b=(331,281,366,271,497,125,70,72,50,214,235,195,387)$ and coefficients sequence of the form \eqref{def:GenerateCoeff} with the modified index set.} 
    \label{fig:Case2_OS}
\end{figure}

Again we use two types of sampling sets. The first set  $\Lambda_3$ is chosen as in \eqref{def:Lambda} so that it formally satisfies the sufficient conditions of Theorem \ref{thm:suffcond}. The redundancy is $q_3 = 3.11$.

The pointwise error between $f_3\indicator_I$ and its approximation $\tilde{f}_3\in \spn\{\psi_{n,l} : \supp (\psi _{n,l}) \cap I \neq \emptyset \}$ is plotted in Figure \ref{fig:Case2_DiffORS_BigSpace_banal}, and the relative errors are reported in Table \ref{tab:Case2_RelErr}. 

\vspace{2mm}

\begin{minipage}{\textwidth}
    \begin{minipage}[c]{0.49\textwidth}
        \centering
        \includegraphics[width=70mm]{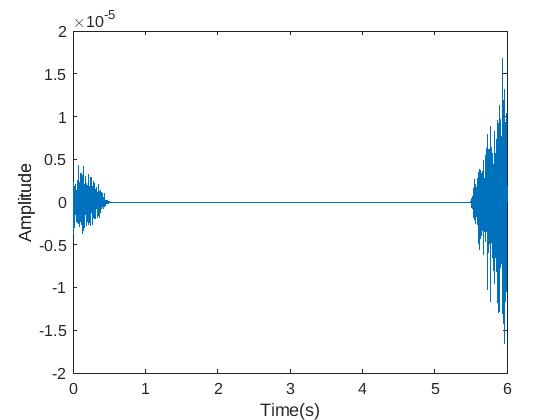}
        \captionof{figure}{Difference between $f_3\indicator_I$ and $\tilde{f}_3 \in \spn\{\psi_{n,l} : \supp (\psi _{n,l}) \cap I \neq \emptyset \}$.}
        \label{fig:Case2_DiffORS_BigSpace_banal}
    \end{minipage}
  \hfill
  \begin{minipage}[c]{0.49\textwidth}
        \centering
        \begin{tabular}{cc}
            \toprule
            \multicolumn{2}{c}{Relative errors}\\
            \midrule
            $E(f,2)$ & $7.264 \times 10^{-7}$\\
            $E(f,\infty) $ & $3.273 \times 10^{-6}$\\
            \bottomrule
        \end{tabular}
        \captionof{table}{Relative errors in $\ell ^2$-norm and $\ell^\infty$-norm.}
        \label{tab:Case2_RelErr}
    \end{minipage}
\end{minipage}

\vspace{2mm}

We see that the reconstruction is highly accurate of the order $\mathcal{O}(10^{-14})$ in the central part of the signal, but oscillates with a pointwise error of a significantly large order $\mathcal{O}(10^{-5})$ at the boundary of $I$. This is not surprising: the samples near the boundary of $I$, in our case $\Lambda_3 \cap ([0,0.5] \cup [5.5,6])$ are needed to recover the coefficients of all terms $\psi_{n,l}$ with $\supp (\psi_{n,l}) \cap I \neq \emptyset$ and $\supp(\psi_{n,l})\not \subseteq  I$. Even if the sufficient conditions of Theorem~\ref{thm:suffcond} are satisfied, there may not be enough samples near the boundary, and the sampling matrix has some rank deficiency. 

To mitigate this effect, we add sampling points outside the interval $I$ and generate a sampling set that satisfies the sufficient conditions \eqref{def:Lambda} on the enhanced interval $\tilde{I} = I \cup [-0.5,0] \cup [6,6.5]$. The resulting oversampling parameter is $3.77$.

\vspace{2mm}

\begin{minipage}{\textwidth}
    \begin{minipage}[c]{0.49\textwidth}
        \centering
        \includegraphics[width=70mm]{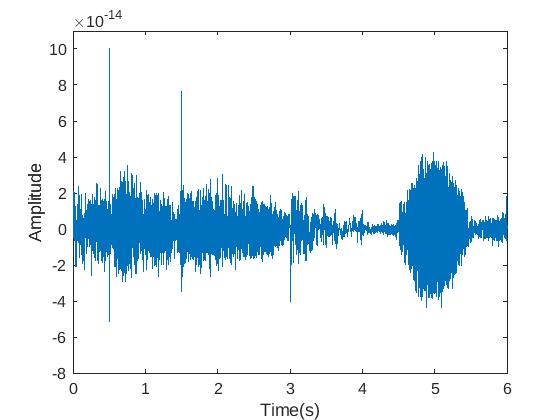}
        \captionof{figure}{Difference between $f_3\indicator_I$ and its reconstruction by sampling in $[-0.5,6.5)$.}
        \label{fig:Case2_DiffORS_BigInt}
    \end{minipage}
    \hfill
    \begin{minipage}[c]{0.49\textwidth}
        \centering
        \begin{tabular}{cc} 
            \toprule
            \multicolumn{2}{c}{Relative errors}\\
            \midrule
            $E(f,2)$ & $6.235\times 10^{-15}$\\
            $E(f,\infty) $ & $1.953 \times 10^{-14}$\\
            \bottomrule
        \end{tabular}
        \captionof{table}{Relative errors in $\ell ^2$-norm and $\ell^\infty$-norm.}
        \label{tab:Case2_RelErr_BigInt}
    \end{minipage}
\end{minipage}

\vspace{2mm}

As shown in Figure \ref{fig:Case2_DiffORS_BigInt} and Table
\ref{tab:Case2_RelErr_BigInt}, the numerical reconstruction on $I$
(not $\tilde{I}$) is now  again accurate to machine precision. The theoretical explanation is simple: by sampling on the extended interval $\tilde{I}$, we replace the given signal $f_3\in \pw $ by a function in the local variable bandwidth space $f_4 \in PW _b^2(g,\tilde{I})$ such that $f_3\big| _I = f_4\big| _I$. We are now in the situation of the first scenario from Section~\ref{sec:51} and try to recover $f_4\in PW_b^2(g,\tilde{I})$ from
samples in $\tilde{I}$. 

Additionally, we investigate the reconstruction from sampling sets $\Gamma_3^\rho$ as in \eqref{def:Gamma_n} for $\rho >0$. Figure \ref{fig:Case2_Err} displays the relative errors in the $\ell^2$-norm and $\ell^\infty$-norms for $\rho = 0.8,0.9,\dots,2.4,2.5$.

\begin{figure}[h!]
    \centering
    \begin{subfigure}[b]{0.45\textwidth}
         \centering
         \includegraphics[width=70mm]{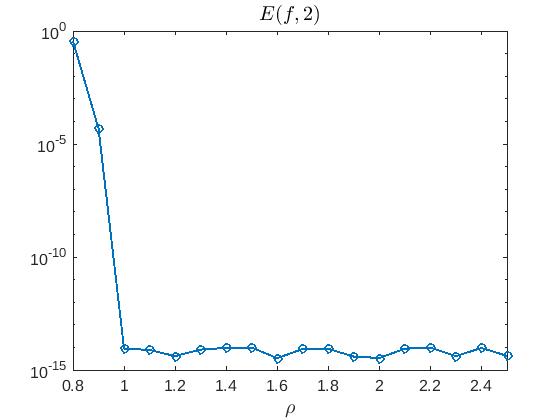}
         \caption{Relative error in the $\ell^2$-norm of the reconstruction by sampling in $[-0.5,6.5)$.}
         \label{fig:Case2_Err2_BigSpace_banal}
     \end{subfigure}
     \hfill
     \begin{subfigure}[b]{0.45\textwidth}
         \centering
         \includegraphics[width=70mm]{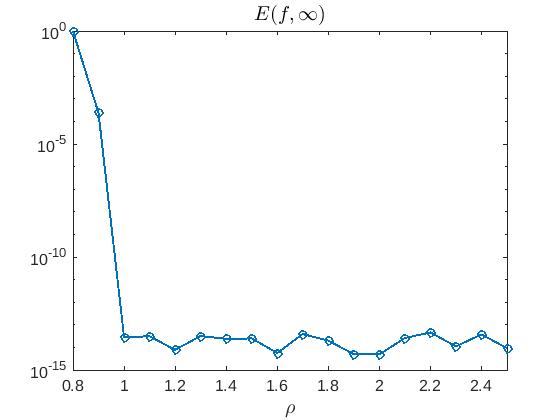}
         \caption{Relative error in the $\infty$-norm by sampling in $[-0.5,6.5)$.}
         \label{fig:Case2_ErrInfinity_BigSpace_banal}
     \end{subfigure}
    \caption{Behavior of the reconstruction in $\spn\{\psi_{n,l} : \supp(\psi_{n,l}) \cap I \neq \emptyset\}$ by sampling in $[-0.5,6.5)$, with increasing parameter $\rho = 0.8,0.9,\dots,2.4,2.5$.}
    \label{fig:Case2_Err}
\end{figure}

The results indicate that the sufficient conditions of Theorem \ref{thm:suffcond} remain pessimistic. For $\rho = 1$, corresponding to an oversampling parameter of $1.33$, the reconstruction already achieves a relative error in the $\ell^2$-norm of the order $\mathcal{O}(10^{-15})$.
 
\subsection{Reconstruction of a chirp.}
In the final example, we aim to reconstruct a linear chirp in the interval $I=[0,6]$, described by the equation 
\begin{equation}\label{eq:chirp}
    p(x) = \sin\Big(\phi(0) + \pi\frac{\omega(T) - \omega(0)}{T}x^2 + 2\pi \omega(0)x\Big),
\end{equation}
where $\phi(0)$ is the initial phase, $\omega(0)$ is the instantaneous frequency at time $0$, and  $\omega(T)$ is the instantaneous frequency at time $T$.
For a linear chirp, the instantaneous frequency function $\omega$ is given by the formula
\begin{equation}
    \omega(t) = \frac{\omega(T) - \omega(0)}{T}t + \omega(0).
\end{equation}
In our simulations, we take the parameters $T = 6$, $\phi(0) = 0$, $\omega(0) = 40 $, and $\omega(T) = 300$ to get the chirp
\begin{equation}
    p(x) = \sin\Big(\frac{130}{3}\pi x^2+80\pi x\Big),
\end{equation}
shown in Figure \ref{fig:Case3_OS}.

\begin{figure}[h]
    \centering
    \begin{subfigure}[b]{0.45\textwidth}
         \centering
         \includegraphics[width=70mm]{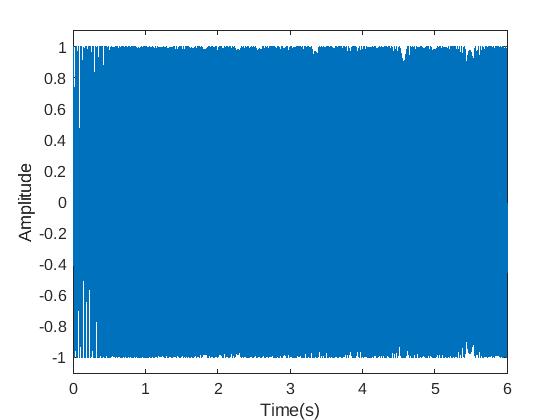}
         \caption{Chirp $p$.}
     \end{subfigure}
     \hfill
     \begin{subfigure}[b]{0.45\textwidth}
         \centering
         \includegraphics[width=70mm]{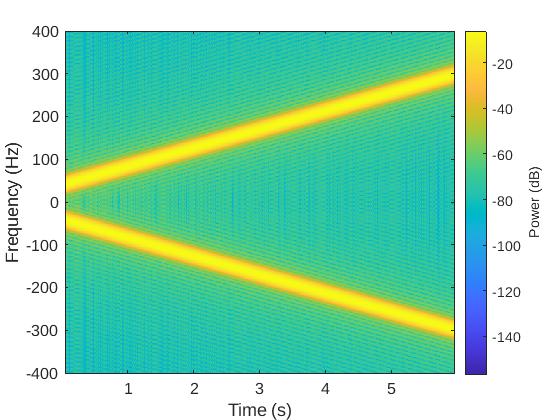}
         \caption{Spectrogram of $p$.}
     \end{subfigure}
    \caption{Chirp $p$ with $\phi(0) = 0$, $\omega(0) = 40$, and $\omega(6) = 300$.}
    \label{fig:Case3_OS}
\end{figure}

In contrast to the previous examples, a general chirp $p$ does not belong to a space $\pw $ of variable bandwidth, but it can be modeled and approximated by a function in $\pw $ with a suitable choice of the local bandwidth $b(n)$.

Since we know the structure of the chirp in advance, we model the
bandwidth sequence $b:\Z  \to \N$, by a linear function 
\begin{equation}
    b(n) \coloneqq \omega \Big(\frac{n}{2} + \frac{1}{2}\Big) +30 = \frac{65}{3}(n+1) + 70, \quad \text{for } n = 0,\dots,12.
\end{equation}
The number $30$ is a safety margin to ensure that the chosen bandwidths contain all the relevant frequencies.
We choose $\spn\{\psi_{n,l} : \supp (\psi _{n,l}) \cap I \neq \emptyset\}$ as the reconstruction space and the dimension of the problem is then
\begin{equation}
    \dim(\spn\{\psi_{n,l} : \supp (\psi _{n,l}) \cap I \neq \emptyset\} ) =  7 + \sum_{n = 0}^{12}b(n) = 2893.
\end{equation}
To mitigate boundary effects, we sample $p$ on the extended interval
$\tilde{I}= [-0.5,6.5]$ and we consider sampling set $\Lambda_5$
satisfying Theorem \ref{thm:suffcond}. 
The oversampling parameter is $q_5 = \# \Lambda_5/\dim(PW _b^2(g,\tilde{I}))= 3.18$. 
The reconstruction algorithm~\eqref{pinv} and \eqref{eq:c5} yields an approximation $\tilde{p}
\in PW_b^2(g,\tilde{I})$ to the original chirp $p$.

\vspace{2mm}
\begin{minipage}{\textwidth}
\hspace{-1cm}
    \raisebox{12pt}{
    \begin{minipage}[c]{0.49\textwidth}
        \centering
        \includegraphics[width=70mm]{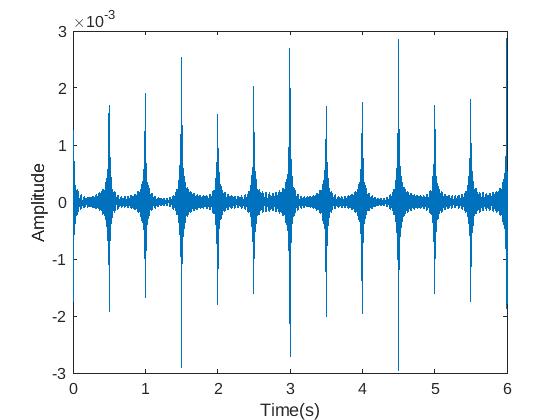}
        \captionof{figure}{Difference between the original chirp $p$ and the reconstruction $\tilde{p}$ using the maximal gap condition and by sampling in $[-0.5,6.5)$.}
        \label{fig:Case3_Diff_ORS}
        \par\bigskip
        \begin{tabular}{cc} 
            \toprule
            \multicolumn{2}{c}{Relative errors}\\
            \midrule
                $E(f,2)$ & $3.926 \times 10^{-4}$\\
                $E(f,\infty) $ & $2.867 \times 10^{-3}$\\
            \bottomrule
        \end{tabular}
        \captionof{table}{Relative errors for $\tilde{p}$.}
        \label{tab:Case3_Diff_OR}
    \end{minipage}
    }
    \hfill
    \raisebox{12pt}{
    \begin{minipage}[c]{0.49\textwidth}
        \centering
        \includegraphics[width=70mm]{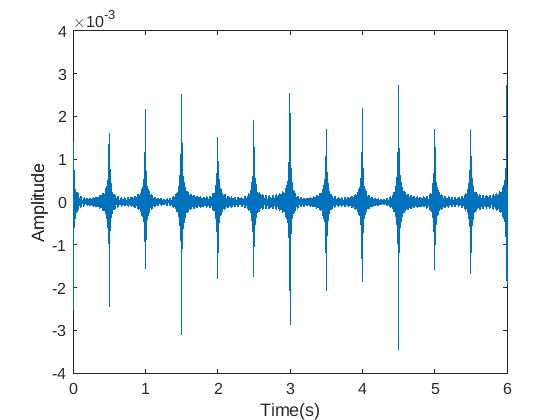}
        \captionof{figure}{Difference between the chirp $p$ and its orthogonal projection $Pp$ onto\\ $\spn\{\psi_{n,l} : \supp (\psi _{n,l}) \cap I \neq \emptyset\}$.}
        \label{fig:Case3_Diff_OP} 
        \par\bigskip
        \begin{tabular}{cc} 
            \toprule
            \multicolumn{2}{c}{Relative errors}\\
            \midrule
                $E(f,2)$ & $3.946 \times 10^{-4}$\\
                $E(f,\infty) $ & $2.731 \times 10^{-3}$\\
            \bottomrule
        \end{tabular}
        \captionof{table}{Relative errors for $Pp$.}
        \label{tab:Case3_Diff_OP} 
    \end{minipage}
    }
\end{minipage} 
\bigskip

Figure \ref{fig:Case3_Diff_ORS} shows the pointwise error between $p\indicator_I$ and  $\tilde{p}\chi_I$. 
Although the error of order $\mathcal{O}(10^{-3})$ seems much larger than in the previous simulations, the result is nearly optimal. The chirp $p$ does not belong to $\pw$ and the best approximation by a function in $\pw $ we can hope for is the orthogonal projection $Pp$. In Figure \ref{fig:Case3_Diff_OP} we therefore plot the pointwise difference between the chirp $p$ and its orthogonal projection onto $\spn\{\psi_{n,l} : \supp (\psi _{n,l}) \cap I \neq \emptyset\}$, where the coefficients $\langle p, \psi_{n,l} \rangle$ arising in $Pp$ are obtained by numerical integration with the trapezoidal method (we used the MATLAB function \verb|trapz|).  
In Tables \ref{tab:Case3_Diff_OR} and \ref{tab:Case3_Diff_OP} the respective relative errors in the $\ell^2$-norm and $\ell^\infty$-norm are presented.

A comparison of the two error plots shows that there is hardly a difference between $p-\tilde{p}$ and $p-Pp$. We conclude that the approximation of $p$ from samples is (almost) the orthogonal projection.
It is interesting to note that the largest errors occur at the half integers $k/2$ for $k \in \Z$. This is to be expected because the chirp is smooth everywhere, whereas the window $g$ from Figure \ref{fig:CosWind} and hence every $f\in \pw $ is continuous, but not differentiable at $\frac{1}{2} \Z $.

Finally, we use the sampling sets  $\Gamma_5^\rho$ from \eqref{def:Gamma_n} and approximate the chirp $p$ from its samples on $\tilde{I}$ for $\rho = 0.9,1,\dots,2.4,2.5$. 

The relative errors in the $\ell^2$-norm and the $\ell^\infty$-norm are displayed in Figure \ref{fig:Case3_Err}. 

\begin{figure}[h!]
    \centering
    \begin{subfigure}[b]{0.45\textwidth}
         \centering
         \includegraphics[width=70mm]{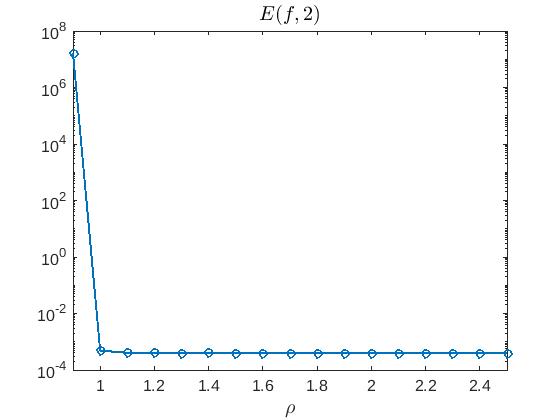}
         \caption{Relative errors in the $\ell^2$-norm of $\tilde{p}$.}
         \label{fig:Case3_Err2}
     \end{subfigure}
     \hfill
     \begin{subfigure}[b]{0.45\textwidth}
         \centering
         \includegraphics[width=70mm]{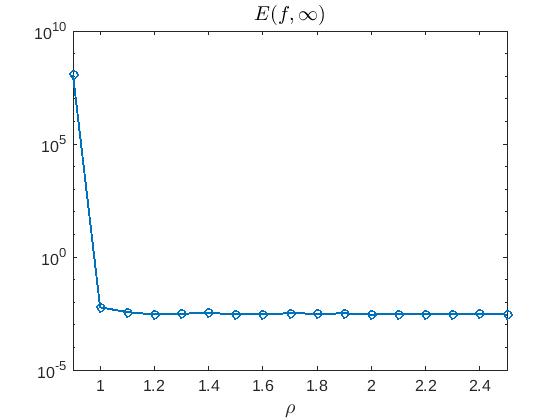}
         \caption{Relative errors in the $\ell^\infty$-norm of $\tilde{p}$.}
         \label{fig:Case3_ErrInfinity}
     \end{subfigure}
    \caption{Behavior of the reconstruction $\tilde{p}$ of the chirp $p$ with increasing  parameter $\rho = 0.9,1,\dots,2.4,2.5$ and by sampling in $[-0.5,6.5)$.}
    \label{fig:Case3_Err}
\end{figure}
The relative errors reach the errors of the orthogonal projection at $\rho = 1$, which corresponds to an oversampling parameter of $1.12$.
Again the relative errors are almost equal to the difference $\|(p-Pp)\big| _I \|_2$, and we obtain an almost perfect reconstruction of $Pp$ rather than $p$ itself.

In conclusion, although  chirps, in particular gravitational waves, are not in a variable bandwidth space, these spaces are convenient parametric models. In this particular application the structure of the signal is known in advance, therefore the use of $\pw $ and of  the associated sampling sets as  in \eqref{def:Gamma_n} appears to be a promising approach for their reconstructions.


\end{document}